\documentclass[11pt,a4paper]{article}

\usepackage[T1]{fontenc}

\usepackage[latin9]{inputenc}
\usepackage{amscd}
\usepackage{mathtools}
\usepackage{lmodern}
\usepackage{amsthm,amsmath,amsfonts,amssymb}
\usepackage{graphicx}
\usepackage{enumerate,enumitem}
\usepackage[colorlinks=true,hyperindex=true]{hyperref}

\let\OLDthebibliography\thebibliography
\renewcommand\thebibliography[1]{
  \OLDthebibliography{#1}
  \setlength{\parskip}{0pt}
  \setlength{\itemsep}{0pt plus 0.3ex}
}

\usepackage[activate={true,nocompatibility},final,tracking=true,kerning=true,spacing=true,factor=1100,stretch=10,shrink=10]{microtype}
\microtypecontext{spacing=nonfrench}

\setlist[enumerate]{topsep = 0.5ex, leftmargin=1cm, itemsep = -2pt}
\setlist[itemize]{topsep = 0.5ex, leftmargin=1cm, itemsep = -2pt}

\theoremstyle{definition}
\newtheorem{theorem}{Theorem}[section]
\newtheorem{lemma}[theorem]{Lemma}
\newtheorem{corollary}[theorem]{Corollary}
\newtheorem{proposition}[theorem]{Proposition}
\newtheorem{definition}[theorem]{Definition}

\newcommand{\abs}[1]{\left\lvert #1 \right \rvert}

\newcommand{\norm}[1]{\lVert #1 \rVert}

\newcommand{\mc}[1]{\mathcal{#1}}
\newcommand{\m}[1]{\mathbb{#1}}

\def\ie{i.e. }

\renewcommand\Re{\operatorname{Re}}
\renewcommand\Im{\operatorname{Im}}

\def\a{\alpha}
\def\b{\beta}
\def\g{\gamma}
\def\G{\Gamma}
\def\d{\delta}

\def\t{\theta}

\def\s{\sigma}

\def\o{\omega}

\def\vare{\varepsilon}
\def\HH{{\mathbb H}}
\def\Chat{\hat{\m{C}}}
\def\eps{\varepsilon}

\def\diam{{\rm diam}}
\def\dist{{\rm dist}}

\def\dd{\,\mathrm{d}}

\newcommand{\ad}[1]{\underline{#1}}

\def\l{W}

\DeclareMathOperator{\SLE}{SLE}

\title{The Loewner energy of loops and regularity of \\ driving functions}

\author{Steffen Rohde\thanks{Department of Mathematics, University of Washington, Seattle, WA 98195, USA, email: rohde@math.washington.edu}  
\, and Yilin Wang\thanks{Department of Mathematics, ETH Z\"urich, Switzerland, email: yilin.wang@math.ethz.ch}}

\begin{document}

\maketitle

\begin{abstract} 
Loewner driving functions encode simple curves in 2-dimensional simply connected domains by real-valued functions. We prove that the Loewner driving function of a $C^{1,\b}$ curve (differentiable parametrization with $\b$-H\"older continuous derivative) is in the class $C^{1,\b-1/2}$ if $1/2<\b\leq 1$, and in the class $C^{0,\b + 1/2}$ if $0 \leq \b \leq 1/2$. This is the converse of a result of Carto Wong \cite{wong2014} and is optimal. We also introduce the Loewner energy of a rooted planar loop and use our regularity result to show the independence of this energy from the basepoint.
\end{abstract}

\section{Introduction}

Loewner \cite{Loewner1923} introduced a conformally natural way to encode a simple curve $\eta$ joining two boundary points of a simply connected plane domain $D$ by a continuous one dimensional real function $W$. This Loewner transform $\eta \mapsto W$ was instrumental in resolving the Bieberbach conjecture \cite{DeBranges1985}, and is the analytic backbone of the Schramm-Loewner evolution SLE \cite{Schramm2000}.

We review the (chordal) Loewner transform in Section \ref{sec_application}. In brief, after replacing $D$ by the upper half-plane $\HH$ via conformal mapping such that $\eta$ joins the boundary points $0$ and $\infty$, we have $W_t=g_t(\eta(t))$ if $\eta$ is parametrized by half-plane capacity and $g_t$ is the hydrodynamically normalized conformal map from $\HH \setminus \eta [0,t]$ onto $\HH.$

Recently, in \cite{friz2015} and \cite{wang2016} the {\it chordal Loewner energy} 
$\int_0^{\infty} \dot W(t)^2 /2 \,dt$ 
of $\eta$ was introduced independently, and basic properties (such as rectifiability)
of curves with finite energy were obtained. The chordal Loewner energy \emph{a priori} depends on the orientation of $\eta$, namely viewed
as a curve from $0$ to $\infty$ or from $\infty$ to $0$. However, the second author \cite{wang2016} proved the direction-independence 
(or reversibility).

In this paper, we generalize the definition of Loewner energy to simple loops on the Riemann sphere 
$\g: \m R \to \Chat$ where $\g$ is continuous, $1$-periodic and injective on $[0,1):$ 
We just observe that the limit when $\vare \to 0$ of the chordal energy of $\g[\vare, 1]$ in the simply 
connected domain $\Chat \setminus \g[0,\vare]$ exists in $[0,\infty]$ (Proposition~\ref{prop:loop_energy_limit_exists}), and define it as the loop Loewner energy of $\g$ rooted at $\g(0)$.
Note that circles have loop energy $0$.
Intuitively, the loop energy measures how much the Jordan curve $\g[0,1]$ differs from a circle seen from the root $\g(0)$, 
in a M\"obius invariant fashion. 
The loop Loewner energy generalizes the chordal Loewner energy: Indeed, if we apply $z \mapsto z^2$ to a chord $\eta$ from $0$ to $\infty$ in 
$\m H$, the positive real line together with the image of~$\eta$ forms a loop~$\g$ through $\infty$. 
It is clear that its loop energy rooted at $\infty$ (\ie we parametrize the loop such that $\g (0) = \infty$) equals the chordal 
energy of~$\eta$.

Note also that the loop energy neither depends on any increasing reparametrization of $\g$ fixing the root, nor on the direction of parametrization. The latter fact basically comes from the chordal reversibility, which can be used to show that $\tilde \g (t) = \g(1-t)$ has the same energy as $\g$ (see Section~\ref{sec_loop_energy} for details).
But it depends a priori on the root $\g(0)$ where the limit is taken, not only on the Jordan curve $\g[0,1]$. However, our first main result
states:

\begin{theorem} \label{thm_main_2} 
The loop Loewner energy is root-invariant.
\end{theorem}

In particular, this result shows that the loop Loewner energy is a conformal invariant on the set of {\it unrooted} loops on the Riemann sphere,
which attains its minimum $0$ only on circles.

In our proof of the root-invariance, we approximate the curve by well-chosen regular curves and are led to the following question:
What can we say about the relation between the regularity of the driving function and the regularity of the curve? 
Prior to this work, only one direction was well understood. 
Slightly imprecisely, the following results state that $C^\a$ driving functions generate $C^{\a + 1/2}$ curves for $\a > 1/2$, where $C^{\a}$ is understood with the usual convention as $C^{n ,\b}$, where $n$ is the integer part of $\a$ and $\b = \a - n$ (see Section~\ref{sec_prelim}). 
More precisely: 

\begin{theorem}(\cite{wong2014})\label{thm_wong}
If $\b \in (0,1/2]$ and $W \in C^{0, 1/2+\b}([0,T])$, then the Loewner curve $\eta$ in $\m H$ generated by $W$ is a simple curve of class $C^{1,\b}$ if reparametrized as $t \mapsto \eta(t^2)$.
    If $W \in C^{1,\b}$, the curve is in $C^{1,\b+1/2}$ (weakly $C^{1,1}$ when $\b = 1/2$).      
\end{theorem}

We will comment on the choice of parametrization later on. Similarly,

\begin{theorem}[\cite{wong2014} and \cite{lindtran2014regularity}] \label{thm_lind_tran}
If $\a > 3/2$ and $W \in C^{\a}$, then $W$ generates a simple curve of class $C^{\a+1/2}$ if $\a +1/2 \notin \m N$, and in the Zygmund class $\Lambda_*^{\a -1/2}$ otherwise.
\end{theorem}

The Zygmund class $\Lambda_*^{\a-1/2}$ contains the class $C^{\a+1/2}$. 
In the other direction, one can ask about the regularity of the driving function given the regularity of the curve.
Here Earle and Epstein proved the following result 
using a local quasiconformal variation near the tip of the curve: 
\begin{theorem}[\cite{EE2001}]
  If $n \in \m Z$, $n \geq 2$ and $\eta \in C^n$, then its driving function is $C^{n-1}$ on the half-open interval $(0,T]$. 
\end{theorem}

They stated the result in the radial setting, but using a change of coordinate it is not hard to see that the regularity of the driving function remains the same in the chordal case. Their result precedes the work of Wong, Lind and Tran, which in turn supported the natural conjecture that $C^{\a+1/2}$ curves should have $C^{\a}$ driving functions when $\a > 1/2$.  

The second main result of this paper is a proof of this conjecture in the case $1/2 < \a \leq 3/2$.
It is the converse of Theorem~\ref{thm_wong} when neither $\a$ nor $\a + 1/2$ is an integer. 
We will discuss the remaining cases of higher regularity in Section~\ref{sec_comments}. 

\noindent {\bf Conventions:} 
We say that an (arc-length parametrized) simple arc $\g: [0,S]\to \m C \setminus \m R_{>0}$ of regularity at least $C^1$  is \emph{tangentially attached} to $\m R_+$ if $\g (0) = 0$, and the right-derivative $\g' (0) = -1$.
In this paper, the curve $\g$ is always at least $C^1$ and arc-length parametrized (we use the variable $s$). 
Loewner driving functions are defined with respect to capacity parametrization (we use the variable $t$).
We use $\eta$ for Loewner curves in $\m H$, in particular for $\sqrt \g$, where $\sqrt{\cdot}$ on $\m C \setminus \m R_+$ is taking values in $\m H$.
 Let $T$ be the half-plane capacity of $\sqrt {\g[0,S]}$.

\begin{theorem}\label{thm_main} Let $0<\b\leq 1$, and $\g$ be a $C^{1,\b}$ simple arc tangentially attached to $\m R_+$.
   The driving function $W$ of $\sqrt \g$ has the following regularity on the closed interval $[0,T]$:
   \begin{itemize}
   \item $C^{0, \b+1/2}$ if $0<\b < 1/2$;  
   \item weakly $C^{0,1}$, if $\b = 1/2$;
   \item $C^{1, \b - 1/2}$ with $\dot W_0 = 0$,  if $1/2 < \b <1$;
   \item weakly $C^{1,1/2}$, if $\b = 1$.
   \end{itemize}
   Their respective norm is bounded by a function of both the local regularity $\norm \g_{1,\b}$ and constants associated with the global geometry of $\g$.
\end{theorem}
The \emph{weak regularity} stands for a logarithmic correction term in the modulus of continuity (see Section~\ref{sec_prelim}). Examples of curves with
bottle-necks easily show that the $C^{\a}$ norm of the driving function 
cannot be bounded solely in terms of the local behavior of~$\g$. 
The sharpness of the Theorem is addressed in Section \ref{subsec_sharpness}.

It is also not hard to deduce from Theorem~\ref{thm_main} that  $C^{1,\b}$ simple loops have finite energy if $\b > 1/2$ (Proposition~\ref{prop_finite_loop}).

Let us comment on the choice of the simply connected domain $\m C \setminus \m R_+$ and subtleties in the parametrization chosen. 
Note that unlike previous results, we study the regularity of the curve on the closed interval $[0,S]$, which requires some regularity of the curve at $0$. This is the reason why we work with curves in the complement of $\m R_+$ rather than in $\m H$. 
In fact, it is trivial but worth mentioning that a simple curve $\g$ is $C^{1,\b}$ on $[0,S]$ and tangentially attached to $\m R_+$ if and only if $\g[0,S] \cup \m R_+$ is a $C^{1,\b}$ simple curve. 
On the other hand, the driving function $\tilde W$ of $\g[0,S] \cup [0,1]$, considered as a chord in the domain $\m C\setminus [1,\infty)$, starts with constant function $0$ (corresponding to the part $[0,1]$) and continues with the driving function of $\g$.
Hence it suffices to study the regularity correspondence between $\g \cup [0,1]$ and $\tilde W$ which is non-trivial only away from the starting point.
Notice that in Theorem~\ref{thm_wong}, the parametrization $t\mapsto \eta(t^2)$ is natural since  in the half-plane setting, $\eta(t)$ is of order $\sqrt t$ for small capacity $t$. 
However,  $t \mapsto t^2$ is smooth, therefore it does not affect the regularity away from $0$.
Therefore, considering regularity correspondence in the domain $\m C\setminus \m R_+$ is more natural than in $\m H$ and Theorem~\ref{thm_wong} can be stated as the implication of the regularity $\tilde W$ to the regularity of $\g \cup [0,1]$ under the usual capacity parametrization. 
Note that, according to the above conventions, in our theorem the smoothness assumption of $\gamma$ is with respect to
the arclength-parametrization, while the stated regularity of $W$ refers to the capacity parametrization.
Since the arclength parametrization has the highest degree of regularity among all parametrizations  that have speed bounded away from zero
(to see this, note that for any parametrization, the arclength function and hence its inverse has the
same regularity as the curve), it follows from Theorem~\ref{thm_main} and Theorem~\ref{thm_wong} that both the arclength and capacity parametrizations of the curve have the same degree of regularity which is also $1/2$ higher than the driving function.

Returning to the strategy of the proof of Theorem \ref{thm_main_2}:
We use concatenated circular arcs to replace a part of the loop and deduce that the energy rooted at two ends of each circular arc are the same if both of them are finite.  
We use Theorem~\ref{thm_main} to show that loops appearing in the surgery are regular enough to have finite energy.
The proof of the general case uses an approximation by minimal energy loops that are of independent interest (Proposition~\ref{prop_optimal_loop}). 
Our proof uses the reversibility of Loewner energy, sometimes implicitly so that we never specify the orientation of loops/arcs and alter freely the orientation.
The reversibility was proved using an interpretation via $\SLE_{0+}$ in \cite{wang2016}, therefore
 our proof of Theorem~\ref{thm_main_2} is not purely deterministic.
 
However, Theorem~\ref{thm_main_2} suggests that a chord in a simply connected domain is better viewed as a part of a loop after conformally mapping the domain to the complement of a circular arc in the sphere as described above, and with regards to the energy, the boundary of the domain plays the same role as the chord.
It also suggests that loop energy has to be a more fundamental quantity. 
Indeed, in a later work \cite{W2} of very different flavor, the second author derived equivalent descriptions of the loop energy connecting to Weil-Petersson class of universal Teichm\"uller space.

The paper is structured as follows: Section~\ref{sec_application} deals with the loop Loewner energy. 
In Section~\ref{sec_chordal_energy}, we briefly recall the results on the chordal Loewner energy that we use, and give the proof of Theorem~\ref{thm_main_2} in Section~\ref{sec_smooth_invariance} and Section~\ref{sec_general_invariance} assuming Theorem~\ref{thm_main}. 
We prove Theorem~\ref{thm_main} in Section~\ref{sec_regularity}, where Section~\ref{sec_reg_h} studies the regularity when a first part $\g[0,s]$ of the curve is mapped-out by the function $h_s$ (Figure~\ref{figure1}). 
It reduces the study to the regularity of the driving function at $0$, details are in Section~\ref{sec_driving_init}. 
We complete the proof in Section~\ref{sec_proof_1}. 
Some comments and possible further development are discussed in the last section.

\section{The Loop Loewner energy}\label{sec_application}

\subsection{Chordal Loewner energy} \label{sec_chordal_energy}
Let $D$ be a simply connected domain in $\m C$, and $a,b$ be two boundary points of $D$.
By a simple curve in $(D,a,b)$ we mean the image of a continuous injective map $\g$ from $[0,1]$ to $\overline{D}$, such that $\g(0 ) = a$ and $\g(0,1) \subset D$. If $\g (1) \in \partial D,$ then we also require that $\g(1) = b$.  
Two curves are considered as the same if they differ only by an increasing reparametrization.

Let us briefly recall the chordal Loewner transformation of a continuous simple curve $\eta$ in $(\m H, 0, \infty)$. It is associated to its driving function $W$ in the following way:
\begin{enumerate}
\item We parameterize the curve in such a way that  the conformal map $g_t$ from $\HH \setminus \eta [0,t]$ onto $\HH$ with $g_t (z)=  z + o(1)$ as $z \to \infty$ satisfies $g_t ( z)= z + 2t / z + o( 1/z)$ (which is the same as saying that the \emph{half-plane capacity} of $\eta[0,t]$ is $2t$, or that $\eta$ is \emph{capacity-parametrized}.)
It is easy to see that it is always possible to reparameterize a continuous curve in such a way. 
\item One can extend $g_t$ continuously to the boundary point $\eta(t)$ and defines $W_t$ to be $g_t ( \eta(t))$.
\end{enumerate}
It is not hard to see that the function $W$ is continuous and $W_0 = 0$. The map $g_t$ is referred to as the \emph{mapping-out function} of $\eta [0,t]$, and the family $(g_t)_{t\geq 0}$ as the \emph{Loewner flow} of~$\eta$. 
The function $W$ fully characterizes the curve through Loewner's differential equation and $W$ is called the \emph{driving function} of $\eta$. 
In fact, consider for every $z \in \m H$ the Loewner differential equation (LDE) in the upper half-plane:
\begin{equation*} 
  \partial_t g_t(z) = 2/(g_t(z)- W_t), 
\end{equation*}
with the initial condition $g_0(z) = z$. The increasing family of the closure of $K_t = \{z \in \m H, \tau(z) \leq t\}$ coincides with the family of $\eta[0,t]$, where $\tau(z)$ is the maximum survival time of the solution. 
And we have also that $g_t : \m H \setminus K_t \to \m H$ is the mapping-out function of $\eta[0,t]$.

\begin{definition}[Chordal Loewner energy]
The \emph{chordal Loewner energy of $\g$ in $(D,a,b)$} is defined as
\[ I_{D,a,b} (\g ) := \frac{1}{2} \int_0^T \dot{\l}(t)^2 d t\]
if $W$ is absolutely continuous,
where  $\l$ is the driving function of the image curve $ \phi (\g)$  under a conformal map $\phi:D \to \m{H}$ with $\phi(a) = 0$ and $\phi(b) = \infty$, and $T$ is the half-plane capacity of $\phi(\g)$ seen from $\infty$. 
The energy is defined to be $\infty$ if $W$ is not absolutely continuous.
\end{definition}
Notice that $T =\infty$ if and only if $\g(1) = b$. 
The choice of the uniformizing map $\phi$ in the above definition is not unique, but they all give the same energy. 
The energy is actually well defined for any chordal Loewner chain, which is the increasing family $(K_t)_{t\geq 0}$ generated by continuous driving function $W$ as above. 
However, it is not hard to see that if the energy is finite and the Loewner chain has infinite capacity, then it is actually a simple curve connecting $a$ to $b$ (see e.g. \cite[Prop.~2.1]{wang2016}). 
Hence we restrict ourselves to simple curves.   
It is an immediate consequence of our absolute continuity assumption that $I_{D,a,b}(\g) = 0$ if and only if $\g$ is contained in the hyperbolic geodesic in $D$ between $a$ and $b$. 
We list some properties of the chordal Loewner energy:
\begin{itemize}
  \item \emph{Conformal invariance.} 
This follows from the invariance of the Dirichlet energy under Brownian scaling, $I_{\m H, 0,\infty} (\g) = I_{\m H, 0,\infty} (a\g)$ for all $a >0$, and allows for the above definition to be independent of the uniformizing map.
\item \emph{Additivity.} Namely
  $$I_{D,a,b} (\g) = I_{D,a,b} (\g [0,\d]) + I_{D \setminus \g[0,\d], \g(\d), b} (\g[\d,1]),$$
  where $0< \d <1$ and we consider $\g[0,\d]$ as a simple curve in $(D,a,b)$ after increasing reparametrization by $[0,1]$, and $\g[\d,1]$ as a simple curve in $(D,\g(\d), b)$ in the same way. We will not explicitly mention such reparametrizations in the sequel, as there is no danger of confusion.
  \item \emph{Regular curves have finite energy.} If $\b > 1/2$, $S< \infty$ and $\g[0,S]$ is an arclength-parametrized $C^{1,\b}$ curve tangentially attached to $\m R_+$, then Theorem~\ref{thm_main} implies that the driving function of $\sqrt \g$ is in $C^{1,\b-1/2}$ on $[0,T]$. Since the capacity $T < \infty$, $\g$ has finite energy in $(\m C \setminus \m R_+, 0, \infty)$. 
  \item \emph{Finite energy implies quasiconformality.} If $\g$ in $(\m H, 0, \infty)$ has finite energy, then it is the image of $i[0,1]$ if $T <\infty$ (or $i \m R_+$ if $T = \infty$) under a quasiconformal homeomorphism $\m H \to \m H$ fixing $0$ and $\infty$.  We say that these curves are \emph{quasiconformal curves},
  see \cite[Prop.~2.1]{wang2016}. It follows essentially from the Lip$(1/2)$ property of the finite energy driving functions \cite{marshall2005loewner} \cite{lind2005sharp}.
  \item \emph{Finite energy curves are rectifiable.} This is proven in \cite[Thm.~2.iv]{friz2015}.
  \item \emph{Corners have infinite energy.} The reason is that finite energy curves in $(\m H, 0,\infty)$ have a vertical tangent at $0$ (see \cite[Prop.~3.1]{wang2016}), while a corner with an opening angle different from $\pi$ generates a curve with non-vertical tangent at $0$ when we map out the portion of the curve up to the corner. More generally, it is not hard to see that finite energy curves are asymptotically conformal \cite[Ch.~11.2]{pomm1992boundary}, using the fact that small energy implies small quasiconformal constant.
  \item \emph{Reversibility.} The chordal Loewner energy is defined in a very directional way, but using an interpretation via $\SLE_{0+}$ and the reversibility of SLE \cite{zhan2008}, the second author proved that the chordal Loewner energy is in fact reversible \cite[Thm.~1.1]{wang2016}:

\begin{theorem}
    For any simple curve $\g \subset D$ connecting $a$ and $b$, 
    \[I_{D,a,b} (\g ) = I_{D,b,a} (\g).\]
    \end{theorem}
    Thus when there is no ambiguity of which boundary points we are dealing with, we simplify the notation to $I_D(\g)$, and view $\g$ as an unoriented curve. 
\end{itemize}

For more background on quasiconformal maps, readers may consult \cite{astala2008elliptic}, \cite{lehto2012univalent} and \cite{lehto1973quasiconformal},
and \cite{lawler2001values}, \cite{lawler2008conformally}, \cite{werner2004st_flour} for background on SLE (introduced by Oded Schramm \cite{Schramm2000}). 
The following corollary is an immediate consequence of the reversibility of Loewner energy, and its counterpart in the Schramm Loewner Evolution setting is known as the commutation relation \cite{Dubedat2007commutation}. The second equality below can also be proved without using reversibility, from explicit computation of the change of driving function, see Proposition~\ref{prop_comm}.

\begin{corollary}[Two-slit Loewner energy] \label{cor_two_slits}
If $\g$ is a simple curve in $(D,a,b)$ and $\eta$ is a simple curve in $(D,b,a)$ such that $\g[0,1] \cap \eta[0,1] = \emptyset$, 
let $\xi$ be the hyperbolic geodesic in $D\setminus \g \cup \eta$ connecting $\g(1)$ and $\eta(1)$. 
We then have
\[I_{D,a,b}(\g \cup \xi \cup \eta) = I_{D,a,b}(\g) + I_{D\setminus \g, b, \g(1)}(\eta) = I_{D,b,a}(\eta) + I_{D\setminus \eta, a, \eta(1)}(\g),\]
and write this value as $I_D(\g \cup \eta)$ without ambiguity. 
\end{corollary}
Combined with the additivity of the energy, the energy of two non-intersecting slits $\g[0,1]$ and $\eta[0,1]$ can be computed by summing up the energies of different pieces that are consecutively attached to previous ones, for instance
\begin{align*}
  &I_{D,a,b}(\g \cup \eta) \\
=& I_{D,a,b}(\g[0,1/3]) + I_{D\setminus \g[0,1/3], b, \g(1/3)} (\eta[0,1]) + I_{D\setminus \g[0,1/3] \cup \eta[0,1], \g(1/3), \eta(1)}(\g[1/3,1]).
\end{align*}

It is not surprising that the Loewner energy strongly depends on the domain. But if we fix the curve, the change of domain entails a change of energy in an explicit way: 
For subsets $A$ and $B$ of a domain $D$, denote
$m^l(D; A, B)$  the measure of Brownian loops (see \cite{lawler2003conformal}) in $D$ intersecting both $A$ and $B$. 
Write 
$H(D;x,y)$ 
for the Poisson excursion kernel relative to local coordinates in the neighborhoods of $x$ and $y$ as defined in \cite[Sec.~3.2]{Dubedat2007commutation},  (see also \cite[Sec.~2.1]{lawler2009part}),
namely the normal derivative of the Green's function $G_D$ using local coordinates.
Note that this number depends on the local coordinates, but the quotients on excursion kernels considered below don't depend on the local coordinates 
if the same neighborhood and the same local analytic coordinates are used for the same boundary point (they all appear once on the denominator and numerator and the excursion kernel changes like a $1$-form at the boundary points when local coordinates change).

Let $H_K$ be a subdomain of $\m{H}$ and assume that they coincide in a neighborhood of $0$ and $\infty$. Let $\g$ be a simple curve in $H_K$.

\begin{proposition}[Conformal restriction {\cite[Prop.~4.2]{wang2016}}]\label{prop_energy_change_finite}
The energies of $\g$ in $(\m{H},0 ,\infty)$ and in $(H_K, 0 , \infty)$ differ by 
$$
   I_{H_K, 0 , \infty}(\g) - I_{\m H, 0, \infty}(\g) = 3\ln\left( \frac{H(H_K; 0, \infty) H(\m H \setminus \g; \g(1), \infty)}{H(\m H; 0 ,\infty) H(H_K\setminus \g; \g(1), \infty)}\right) + 12 m^l(\m H; \g, K)
$$
if $\g(1) \neq \infty$. Otherwise,
$$
   I_{H_K, 0 , \infty}(\g) - I_{\m H, 0, \infty}(\g) = 3\ln\left( \frac{H(H_K; 0, \infty)}{H(\m H; 0 ,\infty) }\right) + 12 m^l(\m H; \g, K).
$$

\end{proposition}
By the conformal invariance of both sides of the above equality, we easily deduce the change of Loewner energy in two general domains which coincide in a neighborhood of the marked boundary points.
\begin{corollary} \label{cor_energy_change_restriction}
   Let $(D,a,b)$ and $(D',a,b)$ be two domains coinciding in a neighborhood of $a$ and $b$, and $\g$ a simple curve in both $(D,a,b)$ and $(D',a,b)$. Then we have if $\g(1) \neq b$,
   \begin{align*}
   I_{D', a, b}(\g) - I_{D, a, b}(\g) = & 3\ln\left( \frac{H(D'; a, b) H(D\setminus \g; \g(1), b)}{H(D; a ,b) H(D'\setminus \g; \g(1), b)}\right) \\
   & 
   +  12 m^l(D; \g, D\setminus D') - 12 m^l(D'; \g, D'\setminus D),
   \end{align*}
   and if $\g(1) = b$,
   \begin{align*} 
   I_{D', a, b}(\g) - I_{D, a, b}(\g)  = 3\ln\left( \frac{H(D'; a, b) }{H(D; a ,b)}\right)  +  12 m^l(D; \g, D\setminus D') - 12 m^l(D'; \g, D'\setminus D).
  \end{align*}
\end{corollary}

From a similar calculation, we also get the difference of the energy of $\g$ in a slit domain $D \setminus \eta$, where $\eta$ grows from the target point of $\g$.

\begin{proposition}[Commutation relation {\cite[Lem.~4.3]{wang2016}}] \label{prop_comm}
Let $\g$ be a simple curve in $(D,a,b)$, and $\eta$ a simple curve in $(D,b,a)$. If $\g \cap \eta = \emptyset$, then
     \begin{align*}
    I_{D \setminus \eta, a, \eta(1)}(\g) - I_{D,a,b}(\g)&
 =\frac{1}{2} \int_0^T \left[\dot{W}_t + 6 U^t\right]^2 - \dot{W}_t^2 \dd t\\ 
 &= 12 m^l(D;\g, \eta) + 3\ln
    \left(
    \frac
    {H(D \setminus \g; \g(1), b) H(D \setminus \eta; a,\eta(1))}
    {H(D \setminus \g \cup \eta; \g(1), \eta(1)) H(D;a,b)}
    \right), \\
    &= I_{D \setminus \g, b, \g(1)}(\eta) - I_{D,b,a}(\eta),
    \end{align*}
    where $T$ is the capacity of $\tilde \g := \varphi (\g)$ seen from $\infty$, and $\varphi$ uniformizes $(D,a,b)$ to $(\m H,0,\infty)$. 
    Let $g_t$ be the mapping-out function of the curve $\tilde \g[0,t]$ parametrized by capacity. The image $\tilde \eta^t := g_t (\varphi(\eta))$ is a slit attached to $\infty$ in $\m H$, and $U^t \in \m R$ is the image of the tip of $1/\tilde{\eta}^t$ under its mapping-out function. 
    \end{proposition}
From the third equality we get again the second equality in the Corollary~\ref{cor_two_slits}.

From now on, we will consider simply connected domains that are complements of simple curves on the Riemann sphere. 
If $\g: [0,1] \to \Chat $ is a simple arc, the domain $\Chat \setminus \g[0,1]$ has two distinguished boundary points, $\g(0)$ and $\g(1)$. 
We will use the shorthand notation $I_{\g}$ for $I_{\Chat \setminus \g[0,1], \g(0), \g(1)}$.
\subsection{Loop Loewner energy} \label{sec_loop_energy}
In this section, we introduce the rooted loop Loewner energy. As we explained in the introduction, it is a natural generalization of the Loewner energy for chords. 
In order to distinguish the different types of energy that we are dealing with, we use the superscript $C$ for chords (\ie $I= I^C$), $L$ for loops and $A$ for arcs.

\begin{definition} A \emph{simple loop} is a continuous 1-periodic function $\g : \m R \to \Chat$, such that $\g (s) \neq \g(t)$, for $0 \leq s < t < 1$. We consider two loops as the same if they differ by an increasing reparametrization.  
\end{definition}

\begin{proposition}\label{prop:loop_energy_limit_exists} Both limits below exist and are equal: 
   \[ \lim_{\vare \searrow 0} I^C_{\g[0,\vare]} (\g[\vare, 1]) = \lim_{\d \searrow 0} I^C_{\g[-\d,0]} (\g[0, 1-\d]) \in [0,\infty].\]
   \end{proposition}
We define the \emph{rooted loop Loewner energy} of a simple loop $\g$ at $\g(0)$ to be this limit, denoted as $I^{L} (\g, \g(0))$.  It is clear that the definition does not depend on the increasing reparametrization fixing $\g(0)$.
Similarly, the energy of $\g$ rooted at $\g(s)$ is
   \[I^L(\g, \g(s)) := I^L(\tilde \g, \tilde \g(0)),\]
   where $\tilde \g$ is $\g$ "re-rooted at $\g (s)$", defined as $\tilde \g (t) = \g (t+s)$.
  
\begin{proof}
The existence follows from 
 \[I^C_{\g[0,\vare]} (\g[\vare, 1]) = I^C_{\g[0,\vare]} (\g[\vare, \vare']) + I^C_{\g[0,\vare ']} (\g[\vare ', 1])  \geq I^C_{\g[0,\vare']} (\g[\vare ', 1]),\]
 if $\vare' > \vare$.  The limit is then an increasing limit as $\vare \to 0$. The proof is the same for $\d \to 0$.
 
For the equality, it suffices to show
 \[ \lim_{\vare \searrow 0} I^C_{\g[0,\vare]} (\g[-1/3, 0] \cup \g [\vare,1/3]) = \lim_{\d \searrow 0} I^C_{\g[-\d,0]} (\g[-1/3, -\d] \cup \g [0,1/3]). \]
The above expressions are two-slit Loewner energies defined in Corollary~\ref{cor_two_slits}.
  In fact, it follows from the reversibility and the additivity of chordal Loewner energy that 
 \[I^C_{\g[0,\vare]} ( \g [\vare, 1])=  I^C_{\g[0,\vare]} (\g[-1/3, 0] \cup \g [\vare,1/3]) + I^C_{\g[-1/3, 1/3]} (\g[1/3, 2/3]).\]
 Now assume $\lim_{\vare \searrow 0} I^C_{\g[0,\vare]} (\g[-1/3, 0] \cup \g [\vare,1/3]) = A <\infty.$
Then $I^C_{\g[-\d,\vare]} ( \g [\vare, 1/3])\leq A$ for all $\vare>0,$ and it follows from the definition of chordal Loewner energy that $I^C_{\g[-\d,0]} ( \g [0, 1/3])\leq A,$ so that
(again from the definition) 
 \[ \lim_{\vare \searrow 0} I^C_{\g[-\d,0]} ( \g [0, \vare]) = 0. \]
 It follows that
\begin{align*}
   & I^C_{\g[-\d, 0]} (\g[-1/3, -\d] \cup \g [0,1/3]) \\ 
   =&  I^C_{\g[-\d,0]} ( \g [0, \vare]) + I^C_{\g[-\d, \vare]} (\g[-1/3, -\d] \cup \g [\vare,1/3])  \\
   = & \lim_{\vare \searrow 0} I^C_{\g[-\d, \vare]} (\g[-1/3, -\d] \cup \g [\vare,1/3])  \\
  = & \lim_{\vare \searrow 0} I^C_{\g[0, \vare]} (\g[-1/3, 0] \cup \g [\vare,1/3]) - I^C_{\g[0, \vare]} (\g[-\d,0]) \\
  \leq & A.
\end{align*}

We conclude that $\lim_{\d \searrow 0} I^C_{\g [-\d, 0]}(\g[-1/3, -\d] \cup \g[0,1/3]) \leq A$, and obtain the equality by symmetry. 
\end{proof}
Similarly, we define the \emph{Loewner energy of a simple arc} (continuous injective) $\eta: [0,1] \to \Chat$ rooted at $\eta(s)$ as follows:
\[I^A (\eta, \eta(s)) = \lim_{\vare \searrow 0} I^C_{\eta[s, s+\vare]}(\eta[0,s] \cup \eta[s+\vare,1]) = \lim_{\d \searrow 0} I^C_{\eta[s-\d, s]}(\eta[0,s-\d] \cup \eta[s,1]).\]

As the definitions suggests, the loop- and arc energies a priori depend strongly on the root, 
but we will prove that they are actually independent of it. We first deal with sufficiently regular loops (for instance in the class $C^{1.5 + \vare}$, $\vare >0$). 
This does not cover all finite energy loops, since there exist such loops which are not even $C^1$, see the last section for a construction of an example. 
We will now show that finite energy loops are quasicircles (images of $S^1$ by  quasiconformal homeomorphisms of $\Chat$). On the other hand,
notice that quasicircles do not necessarily have finite energy. 

\begin{proposition} \label{prop_quasicircle} If $\g$ is a finite energy loop when rooted at $\g(0)$, then $\g$ is a $K$-quasicircle, where $K$ depends only on $I^L(\g, \g(0))$. 
\end{proposition}
\begin{proof}

 \begin{figure}[ht]
 \centering
 \includegraphics[width=0.7\textwidth]{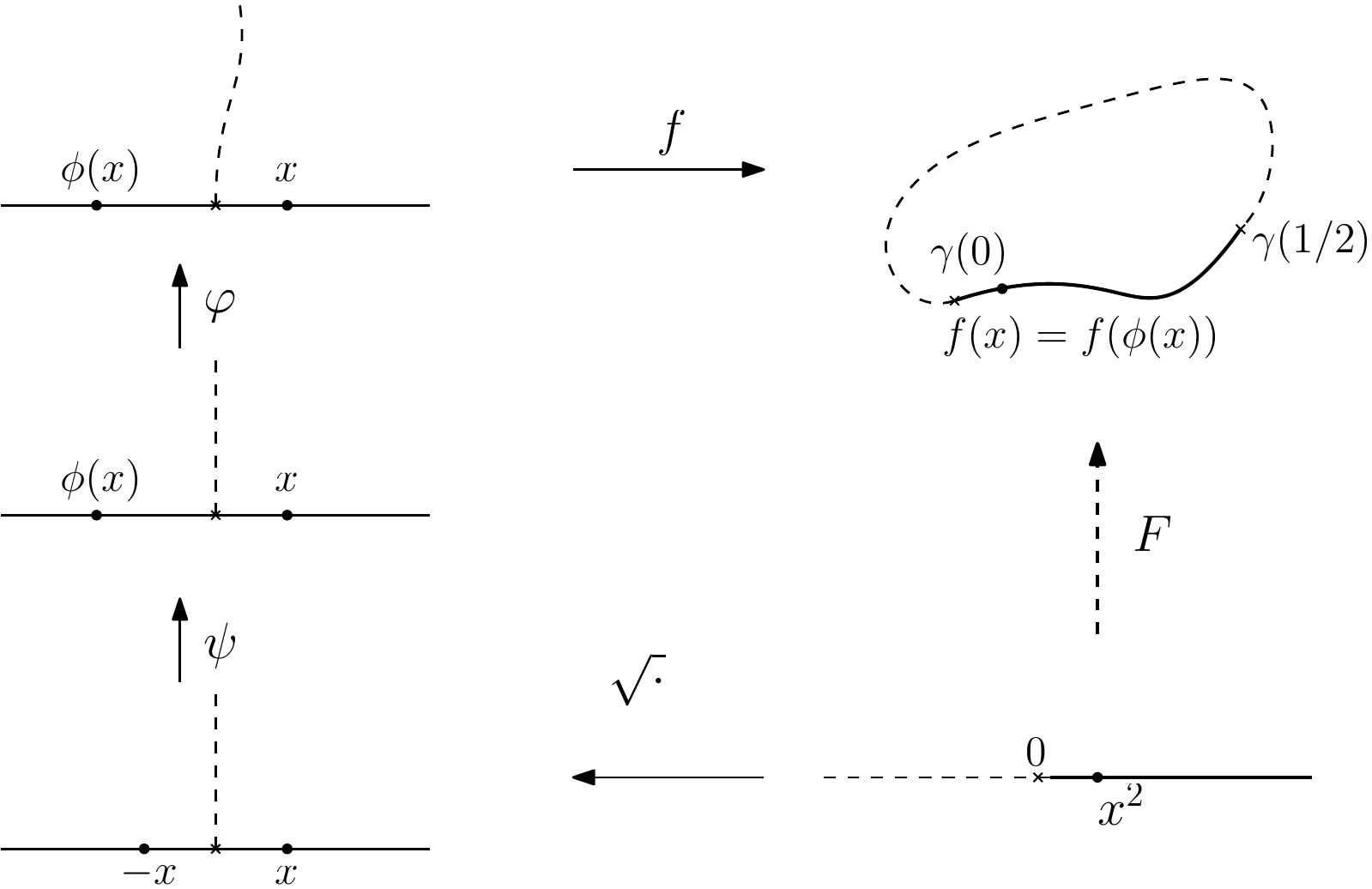}
 \caption{\label{fig_prop_loop} Maps in the proof of Proposition~\ref{prop_quasicircle}. Solid lines are the boundary of domains.} 
 \end{figure}

It is not hard to see from Carath\'eodory's theorem that every uniformizing conformal map $f :\m H\to \Chat \setminus \g[0,1/2]$ extends continuously to $\m R$. 
Thus we may normalize $f$ such that $0$ and $\infty$ are sent to the two tips of $\g[0,1/2]$, say $f(0)=\g(1/2)$ and $f(\infty)=\g(0)$. 
Furthermore, $f$ induces a welding function $\phi$ on $\m R$ that is defined by the property that $\phi(x) = y$ if and only if 
$x= y = 0$ or $f(x) = f(y)$ when $x\neq y$. Thus $\phi$ is a decreasing involution that encodes which points on the real line are identified by $f$ 
in order to form $\g[0,1/2]$. Since $I^A(\g[0,1/2], \g(0)) \leq I^L (\g, \g(0))$, the welding function $\phi$ is an orientation reversing $M$-quasisymmetric 
function, where $M$ depends only on $I^L(\g, \g(0))$: To see this, fix $x > 0$, set $y = \phi(x)$ and let $t>0$ be defined by $\g(t) = f(x)$. 
Then the welding function $\phi$ restricted to $[y,x]$ is the welding function for the slit $\g[t, 1/2]$ in the simply connected domain 
$\Chat \setminus \g[0,t]$. Hence \cite[Prop.~2.1]{wang2016} implies that both inequalities in \cite[Lem.~C]{wang2016} hold on the interval $[y,x]$. 
As we can choose $x$ as large as we want, the inequalities hold on $\m R$ and it follows that $\phi$ is quasisymmetric.

Next, consider the homeomorphism $\psi$ of $\m R$ that sends the symmetric pair of points $x,-x$ to the pair $x, \phi(x)$ for all $x\geq0.$ In other words,
define $\psi(x)=x$ for $x\geq0$ and $\psi(x)=\phi(-x)$ for $x<0.$ Then $f(\psi(-x))=f(\psi(x))$ for all $x.$ It is easy to see, again using
both inequalities in \cite[Lem.~C]{wang2016}, that $\psi$ is quasisymmetric (again with constant depending only on $I^L(\g, \g(0))$).
Any quasisymmetric function that fixes $0$ can 
be extended to a quasiconformal map in $\m H$ that fixes $i\m R_+$ (for instance via the Jerison-Kenig extension, \cite[Thm.~5.8.1]{astala2008elliptic}). 
Denote such an extension again by $\psi.$

Now let $\eta=f^{-1}(\g[1/2,1])$ and note that $I_{\m H}(\eta) \leq I^L(\g, \g(0))$, $\eta$ is a $K$-quasislit by \cite[Prop.~2.1]{wang2016}.
In other words, there exists a $K$-quasiconformal self-map $\varphi$ of $\m H$ fixing $0$ and $\infty$ such that $\varphi (i\m R_+) = \eta$, 
where $K$ depends only on the chordal energy of $\eta$. The restriction of $\varphi$ to $\m R$ is a quasisymmetric function. 
Thus by pre-composing $\varphi$ with a $K$-quasiconformal extension of $\varphi^{-1}$ that fixes $i\m R_+$, we can choose $\varphi$ such that 
$\varphi (x) = x$ for $x \in \m R$. 

Finally, define a quasiconformal homeomorphism of the Riemann sphere that maps the real line to the loop $\g$ as follows: Denote $\sqrt{.}$ the branch of the square-root 
that maps the slit plane $\m C\setminus[0,\infty)$ to $\m H$ and consider the function 
$$F = f \circ \varphi \circ \psi \circ \sqrt{.}\,\, .$$
As a composition of quasiconformal homeomorphisms, it is quasiconformal in $\m C\setminus[0,\infty).$ The negative real line is mapped to $i\m R_+$ under
$\sqrt{.}$, fixed by $\psi,$  mapped to $\eta$ under $\varphi$ and finally mapped to $\g[1/2,1]$ under $f.$ Furthermore, $F$ extends continuously across $\m R_+$:
Indeed, points $x^2\in \m R_+$ split up into the pair $-x,x$ under $\sqrt{.}$, map to the pair $\psi(-x),\psi(x)$, which is unchanged under $\varphi$ and
mapped to a point $f(\psi(-x))=f(\psi(x))$ on $\g[0,1/2]$ under $f.$ Thus $F$ is a homeomorphism of the sphere that is quasiconformal in the complement of 
the real line, and thus quasiconformal on the whole sphere.
\end{proof}

Notice that if $I^L(\g, \g(0)) = 0$, the above proof can be easily modified to prove that $\g$ is a circle ($1$-quasicircle). 

\subsection{Root-invariance for sufficiently regular loops} \label{sec_smooth_invariance}
We first give a sufficient regularity condition for a loop to have finite energy which is a consequence of Theorem~\ref{thm_main}. 

\begin{proposition}\label{prop_finite_loop}
  If $\b > 1/2$, the Loewner energy of a $C^{1,\b}$ simple loop $\g$ rooted at $\g(0)$ is finite.
\end{proposition}
Notice that the regularity does not depend on the choice of root.

\begin{proof}
We first prove that $I^A(\eta, \eta(0)) < \infty$ if $\eta : [0,1] \to \Chat$ is a $C^{1,\b}$ simple arc. To this end, we extend $\eta$ by attaching a small piece of straight segment tangentially at $\eta(0)$, denote the new arc $\eta[-1,1]$, and note that it is again a $C^{1,\b}$ arc. 
From the property of Loewner energy on regular chords that we discussed in Subsection~\ref{sec_chordal_energy}, we know that
\[
I^C_{\eta[-1,0]} (\eta[0,1])<\infty.\]
We have also 
\begin{align*}
I^A(\eta[0,1], \eta(0)) & = I^A(\eta[-1,1], \eta(0)) - I^C_{\eta[0,1]}(\eta[-1,0]) \\
 & = I^A(\eta[-1,0], \eta(0)) + I^C_{\eta[-1,0]}(\eta[0,1]) - I^C_{\eta[0,1]}(\eta[-1,0]) \\
 & \leq 0 + I^C_{\eta[-1,0]}(\eta[0,1]) < \infty.
\end{align*}
In particular,  
$I^A(\g[0,1/4], \g(0)) <\infty$.

Next, we show that $I^C_{\g[0,1/4]} (\g[1/4,1]) <\infty$ which then concludes the proof since
\[I^L(\g, \g(0)) = I^A(\g[0,1/4], \g(0)) + I^C_{\g[0,1/4]}(\g[1/4,1]).\]
Since we are now dealing with an infinite capacity chord, the mere regularity of the driving function is not sufficient to guarantee the finiteness of the energy. Instead, we apply Corollary~\ref{cor_energy_change_restriction} with a domain obtained from a carefully chosen modification of $\g$:
From the first part,
\[I^C_{\g[0,1/4]} (\g[1/4, 3/4]) = I^A(\g[0,3/4], \g(0)) - I^A (\g[0,1/4], \g(0)) < \infty.\]
Similarly $I^C_{\g[0,1/4]} (\g[-1/2, 0]) <\infty$. Let $\tilde \g $ be the simple loop by completing $\g[-1/2, 1/4]$ with the hyperbolic geodesic connecting $\g(-1/2)$ and $\g(1/4)$ in the complement of $\g[-1/2, 1/4]$, such that $\tilde \g(x) = \g(x)$ for $x \in [-1/2,1/4]$ (see Figure~\ref{fig two props}). From the reversibility of the chordal Loewner energy,
\begin{align*} 
   I^C_{\tilde \g [0,3/4]} (\tilde \g [3/4,1]) &= I^C_{\tilde \g [0,1/4]} (\tilde \g [1/4,1]) - I^C_{\tilde \g [0,1/4]} (\tilde \g [1/4,3/4]) \\
   &\leq I^C_{\tilde \g [0,1/4]} (\tilde \g [1/4,1])  \\
   & =  I^C_{\tilde \g [0,1/4]} (\tilde \g [1/2,1]) <\infty.
\end{align*}
Since $\tilde \g$ differs from $\g$ only on the part of the loop parametrized by  $[1/4,1/2]$, the domain $\Chat \setminus \tilde \g[0,3/4]$ coincides with $\Chat \setminus \g[0,3/4]$ in a neighborhood of the two marked boundary points $\g(0)$ and $\g(3/4)$. We can apply Corollary~\ref{cor_energy_change_restriction} to show
\begin{align*}
    I^C_{\tilde \g [0,3/4]} (\g [3/4,1]) - I^C_{\g [0,3/4]} (\g [3/4,1]) < \infty.
\end{align*}
Indeed, since $\g[3/4,1]$ is at positive distance to both $\g[1/4,1/2]$ and $\tilde \g[1/4,1/2]$, the Brownian loop measure term is finite, and the excursion kernel term is always finite. Hence 
\[I^C_{\g[0,1/4]}(\g[1/4, 1]) = I^C_{\g[0,1/4]}(\g[1/4, 3/4]) + I^C_{\g[0,3/4]}(\g[3/4, 1]) <\infty,\]
which concludes the proof.
\end{proof}
In particular, any loop formed by concatenating finitely many circular arcs
has finite energy if and only if any two adjacent arcs have the same tangent at their common point: Indeed, it is easy to check that such a loop is $C^{1,1}$ and any corner with angle different from $\pi$ has infinite energy (see  Section~\ref{sec_chordal_energy}).

 \begin{figure}[ht]
 \centering
 \includegraphics[width=0.8\textwidth]{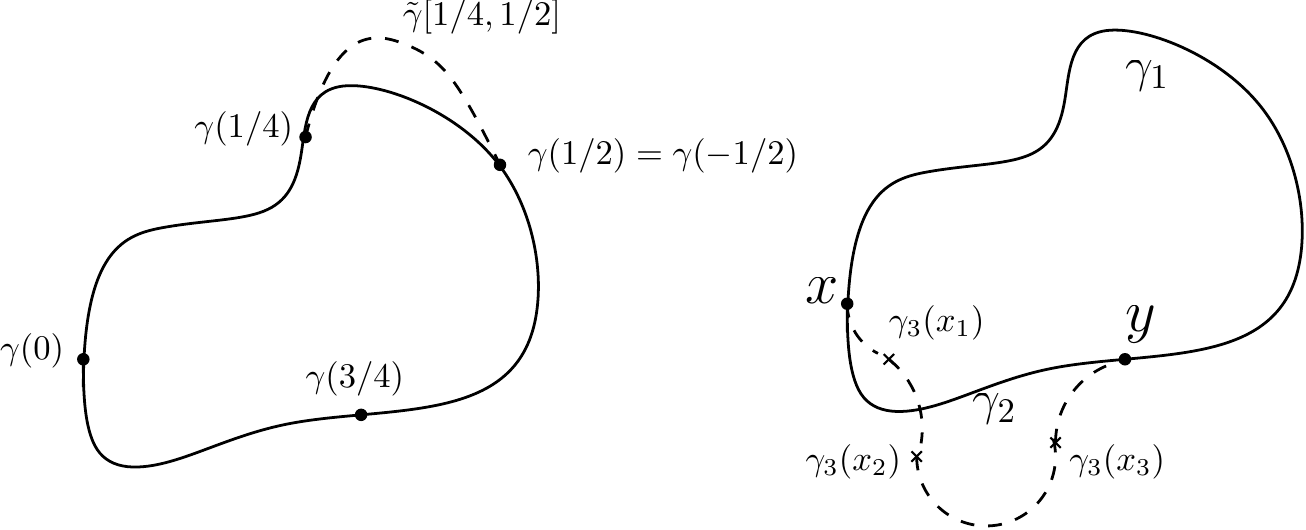}
 \caption{\label{fig two props} Illustrations of the surgeries made in the proof of Proposition~\ref{prop_finite_loop} (left) and Propsition~\ref{prop_regular_invariant} (right). Left: $\tilde \g$ is the loop obtained from replacing $\g[1/4, 1/2]$ by the hyperbolic geodesic in the complement of $\g[-1/2, 1/4]$. Right: $x$ and $y$ separates the solid loop into $\g_1$ and $\g_2$, $\g_3$ is formed by concatenation of circular arcs and replaces $\g_2$ in the proof.} 
 \end{figure}

\begin{proposition}\label{prop_regular_invariant} 
If $\b >1/2$, the Loewner energy of a $C^{1,\b}$ simple loop $\g$ is independent of the root. 
\end{proposition}
\begin{proof} 
   Two distinct points $x, y \in \g$ separate $\g$ into two arcs which we denote by $\g_1$ and $\g_2$.  The additivity gives 
   $$I^L(\g, x) = I^A(\g_1, x) + I^C_{\g_1} (\g_2)$$
   and similarly
   \[I^L(\g, y) = I^A(\g_1, y) + I^C_{\g_1} (\g_2).\]
   Since $I^L(\g,x)$ and $I^L(\g,y)$ are finite, 
   it suffices to prove the equality of the arc Loewner energy on the right hand side.
   
   We complete $\g_1$ by another arc $\g_3$ to form a loop with continuous tangent (see Figure~\ref{fig two props}), where $\g_3 [0,1]$ is a finite concatenation of circular arcs: there exists a sequence $0 = x_0< x_1 <\cdots<x_n = 1$, such that $\g_3[x_i, x_{i+1}]$ is an circular arc for every $i$ (we consider segments as circular arcs).
   
   We give  an explicit construction of $\g_3$: 
   since $\g_1$ is a $C^{1,\b}$ arc, we can first construct a simple, piecewise linear arc $\tilde \g_3$ that connects two tips of $\g_1$, being tangent to $\g_1$ at tips and contained in $\Chat \setminus \g_1$. 
   Then replace each corner of $\tilde \g_3$ by a (very) small circular arc smoothing out the corner without intersecting other parts of the loop.

    Tangentially concatenated circular arcs form a $C^{1,1}$ arc therefore the new loop has finite energy by Proposition \ref{prop_finite_loop}. 
    The above energy decomposition tells us
   \begin{align*} I^A(\g_1, x)  = I^A(\g_1, y)
   &\iff I^L(\g_1 \cup \g_3, x)  
   = I^L(\g_1 \cup \g_3, y)\\ 
   & \iff I^A(\g_3,x) = I^A(\g_3,y).
   \end{align*}
We know that for every circular arc $\eta[0,1]$, the arc energy $I^A(\eta, \eta(s)) = 0$ for all $s \in [0,1]$. It is in particular root-invariant. Hence, for $0 \leq i \leq n-1$,
\begin{align*}
I^A(\g_3[0,1],\g_3(x_i)) &= I^A(\g_3[x_i,x_{i+1}], \g_3(x_i)) + I^C_{\g_3[x_i,x_{i+1}]} (\g_3[0,x_i] \cup \g_3[x_{i+1},1]) \\
& = I^A(\g_3[x_i,x_{i+1}], \g_3(x_{i+1})) + I^C_{\g_3[x_i,x_{i+1}]} (\g_3[0,x_i] \cup \g_3[x_{i+1},1]) \\
&= I^A (\g_3[0,1],\g_3(x_{i+1})).
\end{align*}
Hence
$$I^A(\g_3,x) = I^A(\g_3, \g_3(0)) = I^A(\g_3, \g_3(1)) = I^A(\g_3, y),$$
which concludes the proof.
\end{proof}

\subsection{Root-invariance for finite energy loops} \label{sec_general_invariance}

We are now ready to prove the general root-invariance of the loop Loewner energy.  We start with the lower-semicontinuity of the loop Loewner energy.

\begin{lemma} \label{lem_lower_semicontinuity} Let $(\g_n: [0,1] \to \Chat)_{n \geq 0}$ be a family of simple loops such that $\g_n(k/2) = \g_0(k/2)$ for $k = 0,1$. If there exists 
a simple loop
$\g$ such that $\g_n$ converges uniformly to $\g$, then 
$$\liminf_{n\to \infty} I^L(\g_n, \g_n(0)) \geq I^L(\g, \g(0)).$$
\end{lemma}

\begin{proof} Without loss of generality, we assume that
$$\liminf_{n\to \infty} I^L(\g_n, \g_n(0)) = A < \infty,$$
and $\sup_{n \geq 0} I^L(\g_n, \g_n(0)) = B <\infty$.

For every $0 < \vare <1/4$,
consider the family of uniformizing conformal maps $(\psi_n)_{n\geq 0}$, where $\psi_n$ maps $\Chat \setminus \g_n[0,\vare]$ to $\m H$, sending the two boundary points $\g_n(\vare)$ and $\g_n(0)$ to $0$ and  
$\infty$, respectively, and the interior point $\g_n(1/2) = \g(1/2)$ to a point of modulus $1$. 
Let $\eta_n (s)$ denote the image in $\m H$ of $\g_n (s)$ under $\psi_n$. The curve $\eta_n$ is a chord in $\m H$ connecting $0$ and $\infty$, parametrized by $[\vare,1]$. Similarly, we define $\psi$ and $\eta$ corresponding to $\g$. 

By the definition of loop Loewner energy, 
$$I^C_{\g_n[0,\vare]} (\g_n[\vare,1]) = I^C_{\m H} (\eta_n) \leq B, $$
so that 
all $\eta_n$ are quasiconformal curves with a fixed constant $K$ depending only on $B.$

By the Carath\'eodory kernel theorem, $\psi_n^{-1}$ converges uniformly on compacts of $\m H$ to $\psi^{-1}$.
In fact, since the $\g_n$ are uniformly locally connected, the convergence of $\psi_n^{-1}$ is uniform (with respect to the spherical metric) by \cite[Cor. II.2.4]{pomm1992boundary}. It follows that $\eta_n$,
viewed as $[\vare, 1]$ parametrized curves, converge uniformly to $\eta$ on every interval $[\vare, r]$ with $r<1.$
Let $\l_n$ be the capacity-parametrized driving function of $\eta_n$. We claim that $\l_n$ converges uniformly
on compacts
to the driving function of $\eta.$ To see this, notice that by \cite{marshall2005loewner} the $\l_n$ are uniformly H\"older-1/2, with constant only depending on $B.$ By \cite[Thm.~4.1, Lem.~4.2]{lind2010collisions}, every subsequential limit of $\l_n$ is the driving function of a limit of $\eta_n,$ and the only such limit is  the capacity parametrization of $\eta$. 

From the lower semicontinuity of the Dirichlet energy on driving functions we get 
$$\liminf_{n\to \infty} I^C_{\m H}( \eta_n) \geq I^C_{\m H}(\eta) = I^C_{\g[0,\vare]}(\g[\vare,1]), $$ 
which implies the claim
$$A\geq I^L(\g, \g(0)) $$
by letting $\vare$ to $0$,
since
$$ A = \liminf_{n\to \infty} I^L(\g_n, \g_n(0)) \geq \liminf_{n\to \infty} I^C_{\m H}( \eta_n).$$ 
\end{proof}

Next, we will introduce the curves that we will use to approximate a given finite energy loop. They are minimizers of loop energy among all curves that pass through a given collection of points. In Section \ref{isotopy} below, we will discuss a generalization to the setting of isotopy classes of curves.
Let $\ad{z} = (z_0, z_1, z_2,\cdots, z_n)$ be a finite collection of points in $\Chat$, $\mc{L}(\ad{z})$ be the set of Jordan curves passing through $z_0, z_1, \cdots, z_n, z_0$ in order. We say that curves in $\mc{L}(\ad{z})$ are \emph{compatible with $\ad{z}$}.
Define
$$I^L(z_0,\{\ad{z}\}) := \inf_{\g \in \mc{L}(\ad{z})} I^L(\g, z_0).$$

From \cite[Lem.~3.3]{wang2016}  we know that $I^L(z_0,\{\ad{z}\})$ is finite. 
In fact, one can easily construct a loop which is a small circular arc in a neighborhood of 
$z_0,$ has finite chordal energy, and passes through the other points in order.
We will now show that minimizers exist and are weakly $C^{1,1}$ from the regularity of its driving function.
(The mapping-out functions of energy minimizers are derived explicitly in \cite{MRW2017geodesic}, one obtains the regularity directly from it as well.)

\begin{proposition}\label{prop_optimal_loop} There exists $\g \in \mc{L}(\ad{z})$ such that $I^L(\g, z_0) =  I^L(z_0,\{\ad{z}\})$. Moreover, every such energy minimizer $\g$ is at least weakly $C^{1,1}$.
\end{proposition}

\begin{proof} We first prove the existence. When $\ad{z}$ has no more than $3$ points, a circle passing through all points is a minimizer of the energy. 
Now assume that $\ad{z}$ has more than $3$ points. 
Let $(\g_n)$ be a sequence of finite energy loops compatible with $\ad{z}$ whose energy rooted at $z_0$ converges to $I^L(z_0,\{\ad{z}\})$. 
Let $A$ be the supremum of their energies. Then all $\g_n$ are $K(A)$-quasicircles for some constant $K \geq 1$ due to Proposition~\ref{prop_quasicircle}. 
Let $\varphi_n$ be a $K(A)$-quasiconformal map such that $\varphi_n(S^1) = \g_n$ and $\varphi_n(e^{2i\pi k /3}) = z_k$ for $k = 0,1,2$.
We obtain a normal family of quasiconformal maps which converges uniformly on a subsequence to some $\varphi$. 
In particular, along this subsequence, $\g_n$ converges uniformly to $\g=\varphi(S^1)$ viewed as a curve parametrized by $S^1$. 
From Lemma~\ref{lem_lower_semicontinuity}, we have
$$I^L(z_0,\{\ad{z}\}) = \liminf_{n\to \infty} I^L(\g_n, z_0) \geq I^L(\g, z_0).$$
Since $\g$ is compatible with $\ad{z}$, it is a minimizer in $\mc{L} (\ad{z})$.

To obtain the regularity of $\g$, notice that $\g$ has the following remarkable property:

For $i \in \{0,1,\cdots, n\}$, $z_i$ and $z_{i+1}$  split $\g$ into two arcs $a_{i,1}$ and $a_{i,2}$, where $a_{i,1}$ does not contain other points than $z_{i}$ and $z_{i+1}$ (we set $z_{n+1} = z_0$).   
It is not hard to see that $a_{i,1}$ is the hyperbolic geodesic in the complement of $a_{i,2}$: Otherwise we could replace $a_{i,1}$ by the hyperbolic geodesic, since
$$I^L(\g, z_0) = I^A(a_{i,2}, z_0) + I^C_{a_{i,2}} (a_{i,1})$$
by Corollary \ref{cor_two_slits}.
Thus $a_{i,1} \cup a_{i+1,1}$ is a \emph{geodesic pair} in the simply connected domain $D = \Chat\setminus (a_{i,2}\cap a_{i+1,2})$ between the two marked boundary points $z_i$ and $z_{i+2}$ and passing through $z_{i+1}$, namely $a_{i,1}$ is the hyperbolic geodesic in $D \setminus a_{i+1,1}$ between $z_i$ and $z_{i+1}$, and $a_{i+1,1}$ is the hyperbolic geodesic in $D \setminus a_{i,1}$ between $z_{i+1}$ and $z_{i+2}$. 
Such geodesic pairs have been characterized in \cite{MRW2017geodesic}, and we know that either $a_{i,1} \cup a_{i+1,1}$ form a logarithmic spiral at $z_{i+1}$, or it is the energy minimizing chord in $(D,z_i,z_{i+2})$ passing through $z_{i+1}$. 
In \cite{wang2016}, minimizers are identified and by explicit computation, it is not hard to see that their driving function is $C^{1,1/2}$ which implies weak $C^{1,1}$ trace (see \cite[Thm.~5.2]{wong2014}).
Only the latter case is possible for a minimizing loop $\g$ with constraint $\ad{z}$, as the logarithmic spirals have infinite energy as can be seen by using their self-similarity. 
\end{proof}

To keep this paper self-contained, we outline a proof of the classification of geodesic pairs, and refer to \cite{MRW2017geodesic} for details: 
Assume that $\eta_1$ and $\eta_2$ are two curves in a simply connected domain $D$, forming a geodesic pair through a point $A \in D$. Let $B$ be the boundary point of $D$ on $\eta_2$. The pair separates $D$ into two domains $H_{+}$ and $H_{-}$. Let $R_i$ be the conformal reflection 
in $\eta_i$, which is well-defined in $D\setminus \eta_{i+1}$  ($ i\in \m Z_2$).
Define $F (z)  = R_2 \circ R_1 (z)$ in $H_{+}$, and note that $F$ is a conformal automorphism of $H_{+}$ fixing the boundary point $A.$
From the map $F$ one can recover the welding functions of 
$\eta_1$ and of $\eta_2$ 
as follows: 
Let $\varphi$ be a conformal map from $D \setminus \eta_1$ to $\Chat \setminus \m R_-$ such that $\varphi (A) = \infty$, $\varphi(B)$ = 0. Assume without loss of generality that
$\varphi(H_+)= \m H$.
From the geodesic property,  $\varphi (\eta_2) =\m R_+$. 
The map $ g:= \varphi \circ F \circ \varphi^{-1}|_{\m H}$ defined on the upper half-plane is a M\"obius map fixing $\infty$, hence $$g(x) = ax + b,\quad \text{where } a,b \in \m R \quad \text{and } a>0.$$ 
Moreover, if $[-\infty, -t]$ is the image of $\eta_1 \subset \partial H$ under $\varphi$, it is not hard to see that $g|_{[-\infty, -t]}$ is the welding map of $\eta_1$. Indeed, denoting by $\varphi_+$ resp. $\varphi_-$ the restrictions of $\varphi$ to $H_+$ resp. $H_-,$ we have
$$\varphi_+^{-1} (x) = \varphi_-^{-1} \circ g (x), \quad \forall x \in (-\infty, -t].$$
Since the welding determines the curve (up to conformal change of coordinates), it is then not hard to see that we have the following dichotomy: 
\begin{enumerate}
\item $a = 1$ corresponds to the minimal energy curve in $D$ passing through $A$. See \cite[Sec.~3.2]{wang2016} and the simulation by Brent Werness in Figure~\ref{fig_geo_pair}. 
\item $a \neq 1$ corresponds to a geodesic pair with a logarithmic spiral at $A$.
\end{enumerate}

   \begin{figure}[ht]
 \centering
 \includegraphics[width=0.4\textwidth]{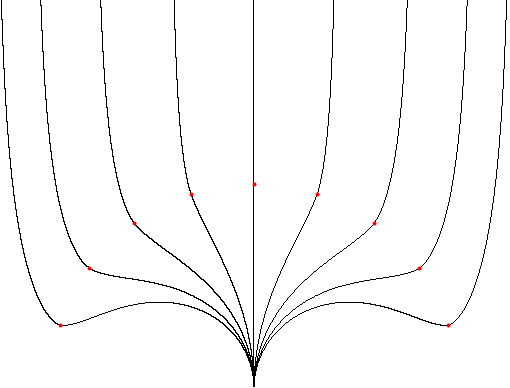}
 \caption{\label{fig_geo_pair} Finite energy geodesic pairs in $\m H$ between $0$ and $\infty$ passing through different points on the unit circle. Simulation by Brent Werness.} 
 \end{figure}

The following  corollary is an immediate consequence of Propositions~\ref{prop_regular_invariant} and \ref{prop_optimal_loop}:

\begin{corollary} \label{cor_minimal_energy} If $\g$ minimizes the energy rooted at $z_0$ among all loops in $\mc{L}(\ad{z})$, then its energy is root-invariant. 
Therefore it also minimizes the energy rooted at $z_k$ for $k \in \{1,\cdots, n\},$ and
$I^L(z_k,\{\ad{z}\}) = I^L(z_0,\{\ad{z}\})$. 
\end{corollary}
Theorem~\ref{thm_main_2} is then an immediate consequence of Corollary~\ref{cor_minimal_energy} and the following 

\begin{proposition} \label{prop_approximation_loop} Let $\g$ be a Jordan curve. The energy of $\g$ rooted at $\g(0)$ is the supremum of $I^L(z_0, \{\ad{z}\})$, where $\ad{z}$ is taken over all finite collections of points on $\g$ which are compatible with $\g$ and have $z_0 = \g(0)$. 
\end{proposition}

\begin{proof}
Let $A$ denote the supremum of such $I^L( z_0,\{\ad{z}\})$. It is obvious that $A \leq I^{L}(\g, \g(0))$. Now we assume that $A < \infty$.

Let $(\ad{z^n})_{n \in \m N}$ be a sequence of increasing $(n+3)$-tuples of points (\ie a point in $\ad{z^n}$ is also in $\ad{z^{n+1}}$), such that the union of points in the sequence is a dense subset of $\g$, $\ad{z^0} = (\g(0), \g(1/3),\g(2/3))$, and the increasing sequence $I^L(z_0,\{\ad{z^n}\})$ converges to $A$.

Let $\g_n$ be a minimizer of the energy (independent of the root due to Corollary~\ref{cor_minimal_energy}) in $\mc{L}(\ad{z^n})$, all of them pass through $\g(0), \g(1/3)$ and $\g(2/3)$.  
Proposition~\ref{prop_quasicircle} tells us that $\g_n$ are all $K$-quasicircle, where $K$ is independent of $n$. 
Let $\varphi_n$ be a $K$-quasiconformal map of $\Chat$ such that $\g_n = \varphi_n(S^1)$ as subsets of $\Chat$. 
By pre-composing with a M\"obius map, we assume that $\varphi_n(\exp( 2 i \pi k/3)) = \g(k/3)$ for all $n \geq 0$ and $k = 0,1,2$. 
Hence $(\varphi_n)_{n\geq 0}$ is a normal family (see e.g.  \cite[Thm.~2.1]{lehto2012univalent}), and a subsequence of $\varphi_n$ converges uniformly to a $K$-quasiconformal map $\varphi$ with respect to the spherical metric. The limiting curve $\g$ passes through all points in $\ad{z^n}$ for every $n$. From the density of points in the union of $\ad{z^n}$, $\varphi(S^1) = \g$.

From Lemma~\ref{lem_lower_semicontinuity}, $I^L(\g, \g(0)) \leq \liminf_{n\to \infty} I^L(\g_n, \g(0)) =  A$ which concludes the proof.
\end{proof}

\section{Proof of Theorem~\ref{thm_main}} \label{sec_regularity}

In this section we prove Theorem~\ref{thm_main}, which was an important tool in our proof of the root-invariance of the Loewner energy. It also is of independent interest, since it gives the optimal regularity of the driving function of an $C^{1,\b}$ curve in most of the cases, see Section~\ref{subsec_sharpness}. 

In Section~\ref{sec_reg_h} we study the regularity of the mapped-out curve, the main results are Corollary~\ref{cor_1_1.5} (for $\b \in (0,1/2]$) and Corollary~\ref{cor_1.5_2} (for $\b \in (1/2, 1]$), which state that up to a Mobius transform in the latter case, the mapped-out curve has the same regularity as the initial curve. 
Therefore it suffices to study the displacement of the Loewner driving function for small times and we see the $1/2$-shift in the regularity (Section~\ref{sec_driving_init}). 
We complete the proof of Theorem~\ref{thm_main} in Section~\ref{sec_proof_1}.

\subsection{Notations}\label{sec_prelim}

Fix $n\in \m N$ and $0< \b \leq 1$. A function $f : I \to \m R$ is $C^{n,\b}$ if there is $C>0$ such that the modulus of continuity $\o(\d; f^{(n)})$ of $f^{(n)}$  on the interval $I$ is bounded by $C \d^\b$ for $\d \leq 1/2$, where 
$$\o (\d; g) = \sup_{\abs{s -s'} \leq \d} |g(s) - g(s')|.$$ 
We denote $\norm{f}_{n,\b}$ the smallest such $C$. When $\b = 0$, the class $C^{n,0}$ corresponds to continuous $f^{(n)}$.

A function $f$ is said to be \emph{weakly} $C^{n,\b}$  if there is $C>0$ such that for all $\d \leq 1/2$,
$$\o(\d; f^{(n)}) \leq C \d^\b \log(1/\d).$$
Sometimes we also write $C^{\a}$ when $\a > 1$, as in Theorem~\ref{thm_lind_tran} above. This stands for $C^{n,\b}$, where $n$ is the largest integer less than or equal to $\a$, and $\b = \a - n$. 

Throughout Section~\ref{sec_regularity}, $\g$ is a \emph{$C^{1,\b}$ arclength-parametrized simple curve tangentially attached to $\m R_+$} for some $\b \in (0,1]$, that is an injective $C^{1,\b}$ function $\g: [0,S] \to \m C \setminus \m R_{+}^*$, such that $\g(0) = 0$, $\g'(0) = -1$ and $\abs{\g'(s)} = 1$ for all $s\in [0,S]$. 
We abbreviate  $\o(\d, \g')$ to $\o(\d)$.

\begin{figure}[ht]
\centering
\includegraphics[width=0.8\textwidth]{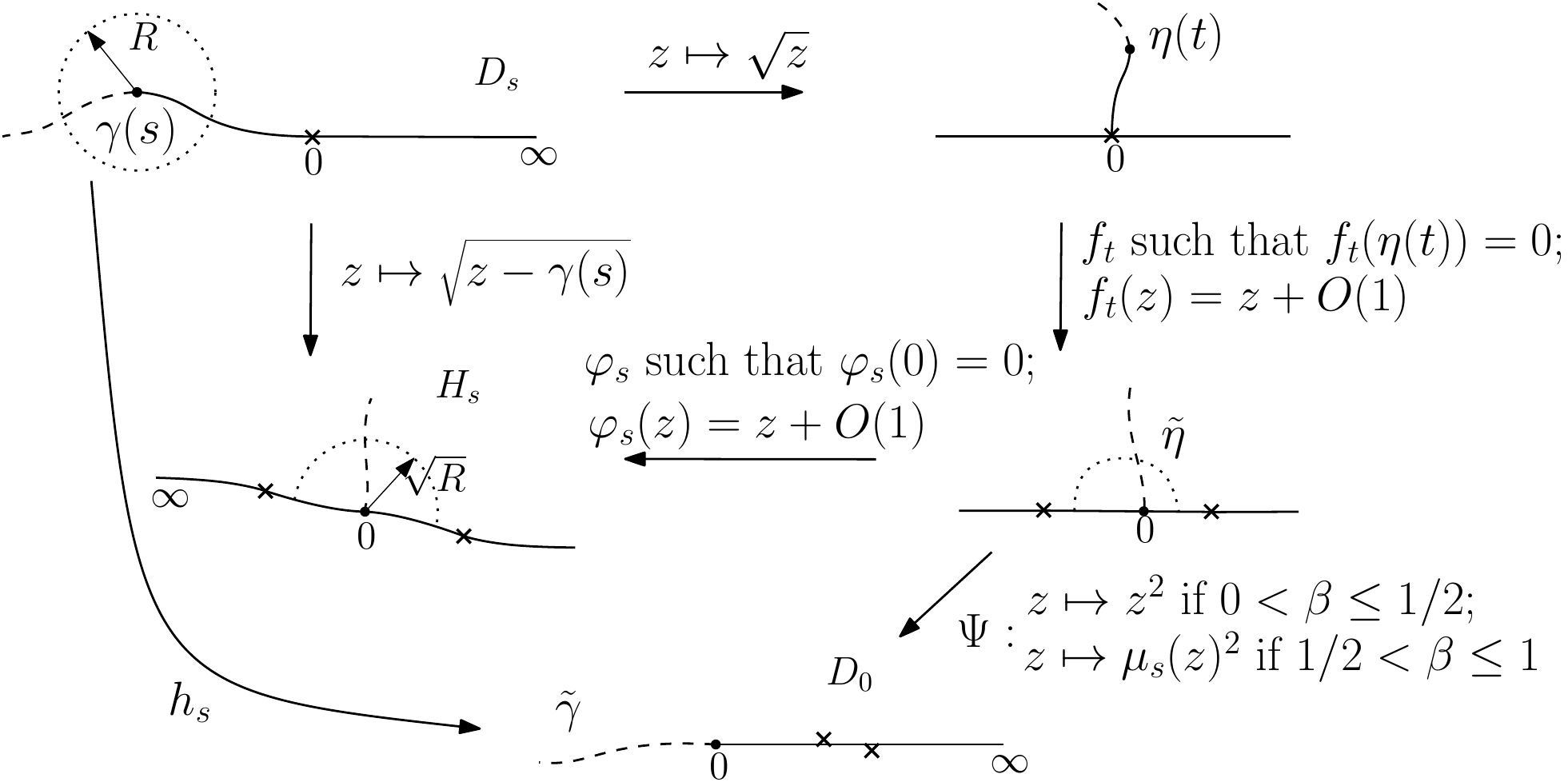}
\caption{\label{figure1} Illustration of different maps considered in Section~\ref{sec_regularity}. We define the map $\Psi$ according to the value of $\b$, and $\mu_s$ is the M\"obius function defined in Corollary~\ref{cor_1.5_2}.} 
\end{figure}

Maps and domains that we use frequently are illustrated in Figure~\ref{figure1}, where 
\begin{itemize}
\item $D_s$ denotes the slitted sphere $\m C \setminus (\g[0,s] \cup \m{R}_+)$;
\item $H_s$ is the image of $D_s$ under $z \to \sqrt{z- \g(s)}$; 
\item $z\mapsto \sqrt{z}$ maps $\g[0,S]$ to a slit $\eta$ in the upper half plane $\m{H}$; 
\item $t= t(s)$ is the half-plane capacity parametrization of $\eta$, that is 
$$cap(\sqrt{\g[0,s]}) = cap(\eta[0,t(s)]) = 2t(s),$$
where the mapping-out function $g_t$ of $\eta[0,t(s)]$ satisfies 
$$g_t(z) = z + 2t(s)/ z + o(1/z);$$
\item $(W_t)_{0 \leq t \leq T}$ is the Loewner driving function of $\eta$ and $T= t(S)$;
\item we also write $\g(-s) = s$ and $W_{-s} = 0$ for $s\geq 0$;
\item $\Psi(z)$ is defined as $z^2$, if $\b \leq 1/2$ and $\mu_s(z)^2$ if $\b > 1/2 $, where $\mu_s$ is a well-chosen M\"obius map (Corollary~\ref{cor_1.5_2});
\item the sphere mapping-out function $h_s(z)$ is given by $\Psi \circ f_t (\sqrt z)$;
\item $\tilde \g$ is the image of $\g[s,S]$ by $h_s$;
\item let $\varphi_s: \m{H} \to H_s$ be the conformal map such that $\varphi_s (0) = 0$ and $\varphi_s(z) = z + O(1)$ as $z \to \infty$.
 \end{itemize}
 The existence and uniqueness of $\varphi_s$ are discussed in Lemma~\ref{lem_u}.
This map is closely related to the centered mapping-out function $f_t : \m{H} \setminus \eta[0,t] \to \m{H}$, that is
\begin{equation}\label{eq_centered_mapping_out}
f_t (z) = g_t(z) - W_t = \varphi_s^{-1} \left (\sqrt{z^2 -\g(s)} \right),
\end{equation}
where $t = t(s)$.
Indeed, it suffices to check that $f_t(\eta(t)) = 0$, and $f_t(z) = z + O(1)$ as $z \to \infty$ which is straightforward.

Regarding the global geometry of $\g$, we assume that there exists $R>0$ such that for all $s \in [0,S]$ and  for all $r \leq R$, the intersection of the disc of radius $r$ centered at $\g(s)$ with $\g(-\infty,S]$ is connected (Figure~\ref{fig_global}).  Intuitively, this rules out bottle-necks of scale less than $R$.
 By taking perhaps a smaller $R$, we assume that $\o(R) \leq 1/5$ and $R \leq 1/2$ (so that our bound for $\o(\d)$ applies for all $\d \leq R$). Using the compactness of $\g[0,S]$, such $R$ can always be found if $\g$ is $C^1$, and we say that $\g$ is \emph{$R$-regular}.

\begin{figure}
\centering
\includegraphics[width=0.6\textwidth]{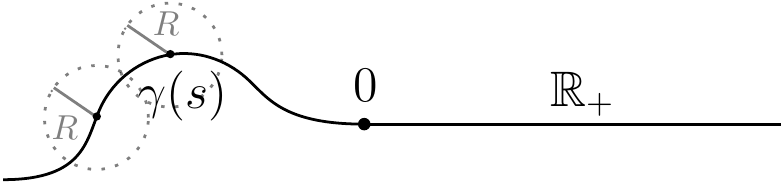}
\caption{\label{fig_global} $C^1$ curve $\g$ without bottle-necks $\leq R$.} 
\end{figure}

\subsection{Regularity of mapped-out curves.} \label{sec_reg_h}

The main goal of this section is to study the regularity of the image of $\g$ under the function $h_s$. It is proven in Corollary~\ref{cor_1_1.5} and Corollary~\ref{cor_1.5_2} that, apart from a minor difference when the regularity is an integer, $\g \cup \m R_+$ and $h_s(\g[s,S]) \cup \m R_+$ are in the same class of regularity modulo a M\"obius transform $\mu_s$ when $\b > 1/2$. Notice that the only non-trivial part of the proof is the regularity of the new curve near the image $0$ of the tip $\g(s)$.

One of our main tools is the Kellog-Warschawski theorem. Roughly speaking, it states that
the conformal parametrization of a smooth Jordan curve (that is, the boundary extension of 
a conformal map of the disc onto the interior of the curve) 
has the same regularity as the arc-length parametrization of the curve, see 
for instance \cite{pomm1992boundary} or \cite{GM}. We also need to keep track of the $C^{1,\b}$-norm of the extension, and
this norm depends not only on the local regularity of the curve, but also on a global property
(roughly speaking, the absence of bottle-necks, which can be quantified for instance by the quasidisc-constant). 
To give a precise statement, define the chord-arc constant of a Jordan curve $\g$ as
$$c_1(\g) = \sup_{z,w\in\g} \frac{\ell(\g(z,w))}{|z-w|},$$
where $\ell$ denotes length and $\g(z,w)$ is the subarc of $\g$ from $z$ to $w$ (in case of a closed Jordan curve, $\g(z,w)$ is the shorter of the two arcs). 
Note that the chord-arc constant $c_1(\g(-\infty,S])$ is bounded in terms of $R,S$ and $\norm{\g}_{1,\b}$: If
$|z-w|$ is small and $\ell(\g(z,w))/|z-w|$ large, then $\g(z,w)\cap D_r(z)$ cannot be connected for suitable $r$.

The following quantitative version is a combination of results from \cite{war1932}(``Zusatz 1 zum Satze 10'', inequality (10,16), p. 440, and ``Zusatz zu Satz 11'', p. 451).

\begin{theorem} \label{war}
If $f$ is a conformal map of the unit disc $\m{D}$ onto the interior domain of a Jordan curve $\g,$ if $D,\ell, c_1, K, \rho$
and $0<\alpha<1$ 
are such that $\diam\ \g\leq D, \ell(\g)\geq\ell$, the chord-arc constant $c_1(\g)\leq c_1,$  $\dist(f(0),\g)\geq \rho$, and
$\o(\d,\arg \g')\leq K \d^\alpha$ for its arc-length parametrization, then there are constants $\mu_1, \mu_2$ and $L$ depending only on $D,\ell, c_1, K, \rho$ and $\alpha$ 
such that
$$\mu_1 \leq |f'(z)| \leq \mu_2 \quad\text{for all}\quad z\in\overline{\m{D}}$$
and 
$$|f'(z)-f'(w)| \leq L |z-w|^\alpha \quad\text{for all}\quad z,w\in\overline{\m{D}}.$$
\end{theorem}

Let us explain the argument in this subsection. 
The sphere mapping-out function $h_s$ is closely related to the conformal map $\varphi_s$, as $h_s(z) = \Psi \circ \varphi_s^{-1}(\sqrt{z-\g(s)})$. 
Lemma~\ref{lem_boundary_regularity} studies the boundary regularity of $H_s$, then Lemma~\ref{lem_u} applies Theorem~\ref{war} to $H_s$ which allows us to compute the angular derivatives of $\varphi_s$ at $0$ in Proposition~\ref{prop_varphi}. Since the curve $\g$ is contained in a cone at $0$, knowing the angular derivatives is enough to compute the regularity of $\eta$ which in turn gives us the regularity of $\tilde \g$ (Corollary~\ref{cor_1_1.5} and Corollary~\ref{cor_1.5_2}).

We start with some trivial but useful estimates on $\g$. 
For every $s \in [0,S]$, $h > 0$, 
\[\g(s + h) - \g(s)  -h \g'(s) = \int_0^h \g'(s+r)-\g'(s) \, dr. \]
Since $\abs{\g'(s+r) - \g'(s)} \leq \o(r) \leq \o(h)$ for $r \leq h$, we have
\begin{equation} \label{ineq_1}
|\g(s + h) - \g(s)  -h \g'(s)| \leq \abs{h}\o(\abs{h}).
\end{equation}
In particular, if $ 0 \leq h \leq R$, then
\begin{equation} \label{ineq_2}
 |\g(s+h) - \g(s)| \geq h - h\o(h) \geq 4h/5. 
 \end{equation}

\begin{lemma}\label{lem_boundary_regularity}
Let $\g$ be a $C^{1,\b}$ curve tangentially attached to $\m R_+$ of total length $S$, $R$-regular. For $s \in (0,S]$, the boundary $\Gamma$ of $H_s$, parametrized by arclength, is a $C^{1,\b}$ curve whose norm is bounded independently of $s$. 
Furthermore, there exists a constant $C >0$, depending only on $R,S$ and $\norm{\g}_{1,\b},$ such that
  $$|\arg (\Gamma'(l)) - \arg (\Gamma'(0))| \leq  C (l^{2\b} \wedge 1)  \quad  \textnormal{ for all }\ l \in \m R,$$
  where $\Gamma( 0 ) = 0$.
\end{lemma}
\begin{proof}

Define
  $$ \tilde \G (r) =  \sqrt{\g(s-r^2) - \g(s)} \quad \text{ for } r\geq 0
  $$ 
and set $\tilde \G(r)=  -  \tilde \G (-r)$ for $r<0.$ Since $\g$ has finite chord-arc constant, $|\g(s-r^2) - \g(s)|$
 is comparable to $r^2$, and consequently
  $$\tilde \G' (r) = \frac{-r \g'(s-r^2)}{\sqrt{\g(s-r^2) - \g(s)}}$$
 is bounded above and away from zero. Since $\G$ is the arc-length parametrization of $\tilde \G$, 
 it easily follows that the modulus of continuity of $\G$ is bounded in terms of the modulus of continuity of $\tilde\G$, $\o_{\G}(r)\leq C\o_{\tilde\G}(C r)$. 
 Hence it suffices to prove the claims of the proposition for $\tilde\G$ instead of $\G.$ 

\noindent
If $\vare>0$ and $r>0,$
  \begin{align*}
   \abs{\tilde \G'(r+\vare) - \tilde \G'(r)} 
  =& \abs{\frac{-(r+\vare) \g'(s-(r+\vare)^2)}{\sqrt{\g(s-(r+\vare)^2) - \g(s)}} - \frac{-r \g'(s-r^2)}{\sqrt{\g(s-r^2) - \g(s)}} } \\
  \leq & \abs{\frac{-(r+\vare) [\g'(s-(r+\vare)^2)-\g'(s-r^2)]}{\sqrt{\g(s-(r+\vare)^2) - \g(s)}} } \\
  & +  \abs{ \frac{-(r+\vare)}{\sqrt{\g(s-(r+\vare)^2) - \g(s)}} + \frac{r}{\sqrt{\g(s-r^2) - \g(s)}}} \\
  \leq &  C \o(2r\vare +\vare^2) + \abs{f(r+\vare) - f(r)},
  \end{align*}
  where $f(r) = r/\sqrt{\g(s-r^2) - \g(s)}$.
By  \eqref{ineq_1} and the aforementioned comparability of $|\g(s-r^2) - \g(s)|$ and $r^2$, we obtain
$$\abs{f'(r)} = \abs{\frac{\g(s-r^2) - \g(s) + r^2 \g'(s-r^2)}{(\g(s-r^2) - \g(s))^{3/2}}} \leq \frac{r^2 \o(r^2)}{(C r^2)^{3/2}} = C_1 \o(r^2)/r.$$
Since $\g$ is a $C^{1,\b}$ curve and for all $\d\leq1/2$, we have $\o(\d) \leq \norm{\g}_{1,\b} \d^\b$ so that
  $$\abs{f'(r)} \leq C_1 \norm{\g}_{1,\b} r^{2\b - 1} \quad\text{for } r \leq1/2.$$
It follows that
  $$\abs{f(r+\vare) - f(r)} \leq C_2 \abs{(r+\vare)^{2\b} - r^{2\b}} \leq C_3 \vare^{2\b \wedge 1}.$$
Letting $r\to0$ we obtain
$$\abs{\tilde \G'(\vare) - \tilde \G'(0)} \leq C\o(\vare^2)+|f(\vare)-f(0)|\leq (C \norm{\g}_{1,\b}+C_2)\vare^{2\b},$$
while for $r<2S$ and $\vare<1/2$ we get
$$\abs{\tilde \G'(r+\vare) - \tilde \G'(r)} \leq  C_4 \vare ^ \b.$$
Direct computation shows that for $r>2S$ we have $\abs{\tilde \G'(r+\vare) - \tilde \G'(r)} \leq  C_5 \vare$,
and we deduce that $\Gamma$ is a $C^{1,\b}$ curve.
\end{proof}

\begin{lemma} \label{lem_u} 
   There exists a unique conformal map $\varphi_s : \m{H} \to H_s$ such that $\varphi_s(0) = 0$ and $\varphi_s(z) = z(1+ o(1))$ as $z\to \infty$. Moreover, $\varphi_s$ extends by continuity to a $C^{1,\b}$ map $\overline{\m{H}} \to \overline{H_s}$, and 
 $$\frac1C\leq \abs{\varphi_s'(r) } \leq C$$
   for all $r \in \m R$ and some constant $C$ depending only on $R,S$ and $\norm{\g}_{1,\b}.$
\end{lemma}

\begin{proof} 
The points $z_0 :=3 i \sqrt{S}\in H_s$ and $-z_0$ have distance at least $\sqrt{S}$ from the boundary $\G$ of $H_s.$ The M\"obius transformation $T_1(z)=(z-z_0)/(z+z_0)$ maps $\G$ to a (closed) Jordan curve $\s=T_1(\G)$. 
We will first show that $\s$ satisfies the assumptions of Theorem \ref{war}, with constants depending only on $R,S$ and $\norm{\g}_{1,\b}$.
Since $\s$ is contained in the image under $T_1$ of the circle of radius $\sqrt{S}$ centered at $-z_0,$ a simple calculation shows that the diameter of $\s$ is bounded above by $12$.  Similarly, the distance $\dist(0,\s)$ is bounded below by the inradius $1/5$ of the image of the circle of radius $\sqrt{S}$ centered at $z_0.$ The length of $\s$ is bounded below since $T_1(\infty)=1$ and $T_1(0)=-1$ are in $\s.$ We already noted that the chord-arc constant $c_1(\g)$ is bounded in terms of $R,S$ and $\norm{\g}_{1,\b}.$
It is an exercise to show that the 
image under the square-root map of a chord-arc curve from $0$ to $\infty$ is chord-arc with comparable constant, so that $c_1(\G)$ is uniformly bounded. It easily follows that $c_1(\s)$ is bounded as well. Finally,
from Lemma~\ref{lem_boundary_regularity} we know that the regularity of $\s$ is $C^{1,\b}$ away from $T_1(\infty)=1$. But from a straightforward computation, we see that $\s$  is also at least $C^{1,\b}$ near $1$.
Thus $T_1(H_s)$ is bounded by a $C^{1,\b}$ Jordan curve. 

Consider the conformal map $f:\m{D}\to T_1(H_s)$ that is normalized by $f(0)=0$ and $f(1)=1,$ and denote 
$p=f^{-1}(-1).$ By Theorem \ref{war}, the derivative of $f$ is bounded above, so that $|p-1|$ is bounded away from zero. Denote $T_2:\m{H}\to \m{D}$ the M\"obius transformation that sends  $\infty$ to $1$, $0$ to $p$, and is furthermore normalized by $T_2(z)=1+c/z +O(1/z^2)$ where $|c|=2|z_0/f'(1)|.$ Then either
$\varphi_s = T_1^{-1}\circ f\circ T_2$ or $-\varphi_s$ is the conformal map from $\m{H}$ to $H_s$ with the desired normalization, and the regularity claims about $\varphi_s$ follow from Theorem \ref{war}.
\end{proof}

Now we are ready to compute the angular derivatives of $\varphi_s$ at $0$. It is not surprising that the highest order that we need to consider is related to the value of $\b$. Heuristically, since the boundary  $\Gamma$ of the domain behaves like a $C^{1+2\b}$ curve at $0$ thanks to Lemma~\ref{lem_boundary_regularity}, one expects that $\varphi_s$ has angular derivatives up to the order $1+2\b$.   
The precise statement is the following:
  
  \begin{proposition}\label{prop_varphi} 
  There exist $L_s > 0$ and $C_1 = C_1(\b, R, S,\norm{\g}_{1,\b})$, such that for all $ 0\leq |x|\leq y \leq 1/2$,
    \begin{align}
   \label{ineq_phi_prime_1}
   \abs{\varphi_s' (x+iy) - \varphi_s'(0)} \leq C_1 y^{2\b}, \quad &  \,\text{ if } \, 0<\b<1/2,  \\
   \label{ineq_phi_prime_1.5}
   \abs{\varphi_s' (x+iy) - \varphi_s'(0)} \leq C_1 y \log(1/y), \quad &  \,\text{ if } \, \b=1/2, \\
   \label{ineq_phi_prime_2}
    \abs{\frac{\varphi_s'' (x+iy)}{\varphi_s' (x+iy)} - L_s} \leq C_1 y^{2\b-1}, \quad &  \, \text{ if } \, 1/2 < \b < 1, \\
    \label{ineq_phi_prime_2.5}
    \abs{\frac{\varphi_s'' (x+iy)}{\varphi_s' (x+iy)} - L_s} \leq C_1 y \log(1/y), \quad &  \, \text{ if } \, \b = 1, 
    \end{align}
    where $\varphi_s$ is defined in Lemma~\ref{lem_u}. Moreover, if $v(r) := \Im \log (\varphi_s'(r))$ for $r \in \m R\setminus\{0\}$, then we have the explicit expression 
 \begin{equation} \label{eq_Ls}
 L_s = \frac{1}{\pi} \int_{-\infty}^{\infty}\frac{v(r) - v(0)}{r^2} dr.
 \end{equation}

  \end{proposition}
  \begin{proof}

We denote the harmonic extension of $v$ to $\overline{\m{H}}$  also by $v$. More precisely, for $x \in \m{R}$ and $y >0$,
\[ v(x+iy) = \frac{1}{\pi} \int_{-\infty}^{\infty} \frac{y}{(r-x)^2 + y^2} v(r) dr.\]
We have $v = \Im \log \varphi_s'$: Indeed, if $u$ is a harmonic conjugate of $v$  on $\m{H}$, then 
$\phi(z) := \log \varphi_s'(z) - (u(z) + iv(z))$ is holomorphic in $\m{H}$, with $\Im(\phi (r)) = 0$ if $r \in \m{R}$. By Schwarz reflection, $\phi$ extends to an entire function with real coefficients. 
Since both $\Im \log \varphi_s'$ and $v$ are bounded in $\m{H}$ (the boundedness of $\Im \log \varphi_s'$ near $\infty$ easily follows from the smoothness of $\s=T_1(\G)$ established in the proof of Lemma~\ref{lem_u}), the imaginary part of $\phi$ is bounded so that $\phi$ is a real constant which we may assume to be zero by adjusting $u.$
Consequently, $u$ and $v$ are the real and imaginary part of $\log \varphi'$.

Since $\varphi_s'(r)$ is bounded away from $0$ and $\infty$, the conformal parametrization of $\partial H_s$ is comparable to the arclength parametrization. 
By Lemma~\ref{lem_boundary_regularity} and Lemma~\ref{lem_u}, there exists $C$ depending on $S$, $R$ and $\norm{\g}_{1,\b}$, such that
$$\abs{w(r)} \leq 
C (r^{2\b} \wedge 1),$$
where $w(r):= v(r) - v(0)$. 
We also have
\begin{align*}
  \partial_x u(x+iy)&= \partial_y v(x+iy) = \frac{1}{\pi} \int_{-\infty}^{\infty}\frac{(r-x)^2 -y^2}{[(r-x)^2 +y^2]^2} v(r) dr \\
  &= \frac{1}{\pi y} \int_{-\infty}^{\infty}\frac{t^2 -1}{(t^2 +1)^2} w(ty + x) dt,
  \end{align*}
  and
  \begin{align*}
  -\partial_y u(x+iy) &= \partial_x v(x+iy) = \frac{1}{\pi} \int_{-\infty}^{\infty}\frac{2y(r-x)}{[(r-x)^2 +y^2]^2} v(r) dr \\
  &= \frac{1}{\pi y}  \int_{-\infty}^{\infty}\frac{2t}{(t^2 +1)^2} w(ty + x) dt .
\end{align*}

For $\b < 1/2$,
we use the bound of  $w(r)$ in the above expressions and obtain for $(x,y)$ with $0 \leq x \leq y \leq  1/2$,
\begin{align*}
\abs{\partial_x u(x + iy)}  \leq & \frac{2C }{\pi y} \abs{ \int_0^{\infty} \frac{|t^2 -1|}{(t^2 +1 )^2} \left(t+\frac{x}{y}\right)^{2\b} y^{2\b}dt + \int_{(1 -x)/y }^{\infty} \frac{|t^2 -1|}{(t^2 +1 )^2} dt} \\
\leq & \frac{2C }{\pi y} \abs{ \int_0^{\infty} \frac{1}{t^2 +1 } (t+x/y )^{2\b} y^{2\b}dt + \int_{(1 -x)/y }^{\infty} \frac{1}{t^2 +1 } dt} \\ 
\leq & y^{2\b -1} \frac{2C }{\pi} \abs{\int_0^{\infty} \frac{1}{t^2 +1 } (t+1 )^{2\b} dt} + \frac{2C }{\pi y} \abs{\arctan\left(\frac{y}{1 - x}\right)} \\ 
\leq & C_2 y^{2\b -1}  + \frac{2C}{\pi (1 - x)} \leq  C_2 y^{2\b -1} + C',
\end{align*}
where $C_2 = 2C  \int_0^{\infty} (t+1 )^{2\b}/(t^2 +1)  dt/\pi$ and $C' = 8C/\pi$.
Similarly, 
\begin{align*}
\abs{\partial_y u(x+iy)} & 
\leq  \frac{2C}{\pi y}  \abs{ y^{2\b }   \int_{0}^{\infty} \frac{2t}{(t^2 +1 )^2} (t+x/y)^{2\b}  dt +  \int_{(1 -x) / y}^{\infty} \frac{2t}{(t^2 +1 )^2}  dt } \\
& \leq  C_3 y^{2\b -1}  + \frac{2C}{\pi y}\frac{y^2}{(1 - x)^2 +y^2} \\
& \leq  C_3 y^{2\b -1} + C' y,
\end{align*}
where $C_3 = 2C \int_0^{\infty} t(t+1 )^{2\b}/(t^2 +1)^2  dt / \pi$. Consequently,

  \begin{align*} 
  \abs{u(x+ iy) - u(0)} 
   \leq &  \abs{\int_0^{y} \partial_r u(ir) dr} 
   + \abs{\int_0^x \partial_r u(r+iy) dr}\\
 \leq & C_3  \int_0^{y} r^{2\b -1} dr   + C' y^2
 + x \left[ C_2 y^{2\b -1}   + C' \right]
  \leq  C_1 y^{2\b},
  \end{align*}
where $C_1$ does not depend on $s$.
Similarly, for the imaginary part,
\begin{align*} \abs{v(x+iy) - v(0)}& = \abs{\frac{1}{\pi} \int_{-\infty}^{\infty} \frac{y}{(r-x)^2 + y^2} ( v(r) - v(0)) dr} \\
&= \abs{\frac{1}{\pi} \int_{-\infty}^{\infty} \frac{1}{t^2 + 1} w(ty +x) dt}\\
&\leq y^{2\b} \abs{\frac{2C}{\pi} \int_{0}^{\infty} \frac{1}{t^2 + 1} (t+1)^{2\b} dt} + \frac{2C}{\pi} \abs{\arctan\left(\frac{y}{1 - x}\right)}\\
&\leq  C_2 y^{2\b}  + C' y \leq  C_1 y^{2\b}.
\end{align*}

 In the case $\b = 1/2$, we need to estimate more carefully, since some of the above integrals diverge. 
Again, for $0 \leq x \leq y \leq 1/2$,
 \begin{align*} 
  \abs{\partial_x u(x+iy) }
  &\leq  \frac{1}{\pi y} \int_{-\infty}^{\infty} \frac{|t^2 -1|}{(t^2 +1 )^2}  \abs{w(ty+x)}dt \\
& \leq \frac{C}{\pi y} \left(  \int_{I(y)} \frac{|t^2 -1|}{(t^2 +1 )^2} |ty +x|  dt +  \int_{\m{R}\setminus I(y)} \frac{|t^2 -1|}{(t^2 +1 )^2} dt \right) \\
& \leq \frac{2 C}{\pi y} \left(  \int_0^{(x+1)/y} y \frac{t+1}{t^2 +1 }   dt +   \int_{(1 -x)/y}^{\infty} \frac{1}{t^2 +1 }   dt \right) \\
& = \frac{C}{\pi} \left[\log(t^2+1) + 2 \arctan (t)\right]_0^{(x+1)/y} + \frac{C}{\pi y} \arctan \left(\frac{y}{1-x}\right) \\ 
& \leq  \frac{2C}{\pi} \log \left(\frac{1}{y}\right) +  \frac{C}{\pi} \log \left(\frac{5}{2} \right) + C + C'\\
& \leq C'' \log(1/y),
  \end{align*}
 where $I(x,y) = [-(x+1)/y , (1 -x )/y]$. For $\partial_y u(x+iy)$, the same bound obtained for $\b <1/2$ also holds for $\b <1$, namely 
  \[ \abs{\partial_y u(x+iy) }\leq C_2 y^{2\b -1} + C'.\]
  Hence  there exists $C_1$ such that for $0 \leq x \leq y \leq 1/2$,
 \begin{align*}
 \abs{u(x+iy) - u(0)} & \leq C_1 y \log(1/y).
 \end{align*}
 A similar calculation also holds for $v$, \ie
 $$
\abs{v(x+iy) - v(0)}  \leq C_1 y \log(1/y). 
$$
\item 
For $1/2 < \b < 1$, $ 0\leq x \leq y \leq 1/2$, we have already seen in the above computation that
  \begin{align*}
  \abs{\partial_x v(x+iy) } = \abs{\partial_y u(x+iy) }  \leq C_3 y^{2\b -1}  +  C' y  \leq C_4 y^{2\b-1} . 
\end{align*}
 We define 
 \begin{equation} 
 L_s = \frac{1}{\pi} \int_{-\infty}^{\infty}\frac{w(r)}{r^2} dr
 \end{equation}
and obtain
\begin{align*}
  \partial_x u(x+iy) - L_s &= \frac{1}{\pi} \int_{-\infty}^{\infty}\left[ \frac{(r-x)^2 -y^2}{[(r-x)^2 +y^2]^2}  - \frac{1}{r^2} \right] w(r) dr \\
  &=\frac{1}{\pi} \int_{-\infty}^{\infty}\left[ \frac{P(r)}{[(r-x)^2 +y^2]^2 r^2}  \right] w(r) dr,
   \end{align*}
where $P$ is a polynomial of degree $3$ with coefficients in $\m R[x,y]$. After the change of variable $r = ty + x$, and set $\xi = x/y$, we get 
  \begin{align*} 
  \abs{\partial_x u(x+iy) - L_s} \leq  & y^{2\b - 1} \frac{C}{\pi} \int_{-\infty}^{\infty}\abs{\frac{\tilde{P}(t, \xi)}{\tilde{Q}(t,\xi)}}  (\abs{t}+1)^{2\b} dt +  \frac{C}{\pi y} \int_{\m R \setminus I(x,y)}\abs{\frac{\tilde{P}(t, \xi)}{\tilde{Q}(t,\xi)}}  dt \\
  \leq & y^{2\b - 1} \frac{C}{\pi} \int_{-\infty}^{\infty}\abs{\frac{\tilde{P}(t, \xi)}{\tilde{Q}(t,\xi)}}  (\abs{t}+1)^{2\b} dt +  \frac{C_5}{y} \int_{(1-x)/y }^{\infty} \frac{dt}{t^3}  \\
  =& y^{2\b - 1} \frac{C}{\pi} \int_{-\infty}^{\infty}\abs{\frac{\tilde{P}(t, \xi)}{\tilde{Q}(t,\xi)}}  (\abs{t}+1)^{2\b} dt + C_6 y,
  \end{align*}
  where $C_5$ and $C_6$ are universal constants. Both $\tilde P$ and $\tilde Q$ have degree $6$ in the second variable, and degree $3$ and $6$ respectively in the first variable.
Since $\xi \in [-1,1]$, and $\tilde P(t,\xi)/ \tilde Q(t,\xi) (\abs{t} +1) ^{2\b}$ can be uniformly bounded  by an integrable function ($ \sim (1+t)^{2\b -3}$), we know that there exists $C_1 = C_1( \b, S,R, \norm{\g}_{1,\b}) >0$ such that 
 \[\abs{\partial_x u(x+iy) - L_s} \leq  C_1y^{2\b - 1},\]
and similarly 
\[\abs{\partial_x v(x+iy)} \leq C_1y^{2\b -1}.\]
In terms of $\varphi_s$, 
\[ \frac{\varphi_s''}{\varphi_s'} (x+iy) = \log(\varphi_s')' (x+iy) = \partial_x u + i\partial_x v.\]
We have thus obtained the bound \eqref{ineq_phi_prime_2}.

The case where $\b = 1$ is similar to the case $\b = 1/2$. Integration of $dt/t$ on the interval 
$I(x,y) = [-(x+1)/y, (1 -x)/y]$ gives the $\log (1/y)$ term. 
\end{proof}

We define $\nabla := \{z = x+iy \in \m H, \, y \leq 1/2 \text{ and } \abs{x} \leq y\}$. 
Let $\g$ be a $C^{1,\b}$ curve.
From the above proposition, it is easy to see that there exists $R_0 >0$ such that for all $s \in [0,S]$, $\sqrt{\g(s+r) -\g(s)} \in \varphi_s(\nabla)$ for all $r \in [0,R_0]$ where the map $\varphi_s$ is as defined in Lemma~\ref{lem_u}.

\begin{corollary} \label{cor_1_1.5} If $0<\b\leq 1/2$,
the image $\tilde{\g}$ of $\g[s,S]$ under the conformal map 
   \[h_s(z) = \left[ \varphi_s^{-1} \left (\sqrt{z -\g(s)} \right )\right]^2, \quad D_s \to D_0\]
   is also a $C^{1,\b}$ curve (weak $C^{1,\b}$ curve if $\b = 1/2$). More precisely, its behavior near $0$ under arclength parametrization is
   \begin{align*}
   \abs{\tilde \gamma'(r) + 1 } &\leq C_2 r^\b  \hspace{2cm}  \text{ if }  0< \b < 1/2, \\
   \abs{\tilde \gamma'(r) + 1 } &\leq C_2 r^\b \log (1/r)  \hspace{1.1cm}  \text{if } \b =1/2,
   \end{align*}
or all $r \leq R_0,$ where $C_2$ is independent of $s$.
   \end{corollary}

\begin{proof} It is obvious that the image of $\g[s + \vare, S]$ under $h_s$ is a $C^{1,\b}$ curve. We only need to check that the limit of $\partial_r h_s(\g(s+r))$ as $r \to 0$ is in $\m{R}_-$, with convergence rate $r^\b$ if $\b <1/2$ and $r^{1/2} \log(1/r)$ if $\b = 1/2$.

  We use the same notation 
  $$\log \varphi_s'(z) = u(z) +i v(z)$$
  as before.  
  Set $\psi := \varphi_s^{-1}$, we have $\psi'(z) = 1/\varphi_s'(\psi(z))$.
  Thus
  $$\psi'(0)^2 \g'(s) = \varphi_s'(0)^{-2} \g'(s)= -\exp(-2u(0)) < 0.$$

  For $ 0< \b< 1/2$ and $z \in \varphi_s(\nabla)$, 
  from \eqref{ineq_phi_prime_1} and the boundedness of $|\varphi_s'|$ we have 
    \[\abs{\psi'(z) -\psi'(0)} = \abs{\frac{\varphi_s'(0) - \varphi_s'(\psi(z))}{\varphi_s'(\psi(z)) \varphi_s'(0)}} \leq C_1 \abs{\psi(z)}^{2\b} \leq  \tilde C_2 \abs{z}^{2\b}\] 
  hence 
  \[\abs{\psi(z) - z\psi'(0)} \leq \tilde C_2\abs{z}^{1+2\b}. \]
  We know that 
  \[\abs{\g'(s+r) - \g'(s)} \leq \norm{\g}_{1,\b} \abs{r}^\b \text{ and } \abs{\g(s+r) - \g(s) - r \g'(s)} \leq \norm{\g}_{1,\b} \abs{r}^{1+\b}. \]

By the definition of $R_0,$ we have $\G_r := \sqrt{\g(s+r) - \g(s)} \in \varphi_s(\nabla)$ for all
$r\leq R_0$ with $s+r\leq S$. For such $r,$ the estimate \eqref{ineq_phi_prime_1} yields
    \begin{align*}
    & \abs{\partial_r (h_s (\g(s+r))) - \psi'(0)^2 \g'(s)} \\
    =& 
    \abs{ \psi \left (\G_r \right ) \psi' \left (\G_r \right ) \g'(s+r) /\G_r- \psi'(0)^2 \g'(s) }\\
    \leq &  \abs{\psi(\G_r) - \psi'(0) \G_r}\abs{ \psi'(\G_r)} 
   /\G_r 
    + 
    \abs{\psi'(\G_r) - \psi'(0)} \abs{\psi'(0)} \\
    & + \abs{\g'({s+r}) - \g'(s)} \abs{\psi'(0)}^2  \\
    \leq & \tilde C_2  \left(\abs{\psi'(\G_r)} {\G_r}^{2\b+ 1} /\G_r + \abs{\psi'(0)} \abs{\G_r}^{2\b} + \abs{\psi'(0)}^2 r^\b  \right)
    \leq  C_2 \abs{r}^\b,
    \end{align*}
    since $\psi'$ is uniformly bounded.
    In particular, $\partial_r(h_s(\g(s+r)))|_{r=0} = -\abs{\psi'(0)}^2$ and $r \mapsto h_s(\g(s+r))$ is a $C^{1,\b}$ function. 
    It is easy to see that $\partial_r(h_s(\g(s+r)))$ is bounded away from $0$ and $\infty$, the above estimate suffices to conclude that $\tilde \g = h_s(\g [s,S])$ is also $C^{1,\b}$ when parametrized by arclength.

    In the case $\b = 1/2$, the argument for the behavior at $h_s(\g(s))$ is the same by using the bounds
    $$\abs{\psi'(z) -\psi'(0)} \leq C_1 \abs{z} \log(1/\abs{z}) \text{ and } \abs{\psi(z) - \psi'(0) z} \leq C_1 \abs{z}^{2} \log(1/\abs{z}) $$
 in the above computation of $\abs{\partial_r (h_s(\g(s+r))) - \psi'(0)^2 \g'(s)}.$
The latter of the two inequalities is obtained from an integration. 
\end{proof}

We now turn to the case $1/2< \b \leq 1$. Let $\mu_s$ be the M\"obius transform $\m H \to \m H$ with $\mu_s (0) = 0$, $\mu_s ' (0) = 1$ and $\mu_s''(0) = L_s$. 
\begin{corollary} \label{cor_1.5_2}
The angular limit as $z \to 0$ of $[\mu_s \circ \varphi_s^{-1}]''/[\mu_s \circ \varphi_s^{-1}]'(z)$ is $0$, with the same rate  of convergence as in Proposition~\ref{prop_varphi}. The image 
$\tilde \g$ of $\g[s,(s+R_0) \wedge S]$ 
under the conformal map 
   \[h_s(z) =\left[ \mu_s \circ \varphi_s^{-1} \left (\sqrt{z -\g(s)} \right )\right]^2\]
   satisfies
   \[\o(\d; \tilde \g') \leq \begin{cases} C_2 \d^\b & \text{ if } 1/2 < \b < 1\\
   C_2 \d^\b \log(1/\d) & \text{ if } \b = 1,
   \end{cases}\]
   where $R_0$ and $C_2$ depend only on $\b, R, S$ and $\norm{\g}_{1,\b}$ (in particular do not depend on $s$).
\end{corollary}

\begin{proof}
  We first check that $[\mu_s \circ \varphi_s^{-1}]''/[\mu_s \circ \varphi_s^{-1}]'(z)$ has angular limit $0$. Again denoting $\psi = \varphi_s^{-1}$, we have 
  \[ 0 = [\varphi_s \circ \psi]''/[\varphi_s \circ \psi]' (z) = [\varphi_s''/\varphi_s' (\psi(z)) ]\psi'(z) + \psi''/\psi'(z) .  \]
   For $1/2<\b<1$ and $z \in \varphi_s(\nabla)$,
  \begin{align*}
   & \abs{[\mu_s \circ \varphi_s^{-1}]''/[\mu_s \circ \varphi_s^{-1}]'(z) } \\
  = &\abs{\mu_s''/\mu_s'(\psi(z)) \psi'(z) + \psi''/\psi'(z)} \\
  = &\abs{\left[L_s + R_1\left(z\right)\right] \psi'(z) -   \varphi_s''/\varphi_s'(\psi(z)) \psi'(z) }\\
  = & \abs{\left[L_s + R_1\left(z\right)\right] \psi'(z) -  \left[L_s + C_1 \left(\abs{\psi(z)}^{2\b -1}\right) + R_2(\abs{z}^{2\b -1}) \right]  \psi'(z) }
  \\
  \leq & C' \abs{z}^{2\b -1},
  \end{align*}
where  $\abs{R_1(z)/z}$ is uniformly bounded on $\varphi_s(\nabla)$, $s\in [0,S]$; $\abs{R_2(z)/z^{2\b -1}} \to 0$ uniformly as $z \to 0$ in $\varphi_s(\nabla)$, and $C' > 0$ does not depend on $s$. It yields the angular limit $0$ with convergence rate as in Proposition~\ref{prop_varphi}. 
  
  The analysis of the behavior of $\abs{\tilde \g' (r)  - \tilde \g'(0)}$ near $0$ is the same as in Corollary~\ref{cor_1_1.5}. But unlike Corollary~\ref{cor_1_1.5},  we need to bound in addition the modulus of continuity of $\tilde \g'$ on a small neighborhood of $0$. To this end, we first estimate the Lipschitz constant of $\phi(z)/z$ where $\phi(z) = \mu_s \circ \varphi_s^{-1}$.
  
  Since $\phi'(z), z\in \varphi_s(\nabla)$ is bounded by a constant independently of $s$, we have 
  $$\abs{\phi''(z)} \leq C'' \abs{z}^{2\b -1}. $$
   Hence for $z, h \in \m C$ such that the segment $[z , z+h] \subset \varphi_s(\nabla)$,
  \[\abs{\phi'(z + h) - \phi'(z)}\leq C'' \int_0^{\abs{h}} (\abs z + u)^{2\b -1} \dd u 
  \leq C'''\abs{h} (\abs z + \abs h)^{2\b -1}. \]
  For $z_1, z_2 \in \varphi_s(\nabla)$ such that $[tz_1, t z_2] \subset \varphi_s(\nabla)$ for all $t\in [0,1]$, 
  \begin{align*}
  \abs{\frac{\phi(z_1)}{z_1} - \frac{\phi(z_2)}{z_2}} \leq 
  & \int_0^1 \abs{\phi'(tz_1) - \phi'(tz_2)} \dd t \\
  \leq & C''' \int_0^1 t^{2\b } \abs{z_1 - z_2} (\abs z + \abs {z_1 - z_2})^{2\b -1} \dd t \\
  \leq & C''' \abs{z_1 - z_2} (\abs z + \abs {z_1 - z_2})^{2\b -1}. 
  \end{align*}
  Now the analysis of $\tilde \g'$  is straightforward: write $\G_r := \sqrt{\g(s+r) -\g(s)}$ for simplicity,
   \begin{align*}
    \partial_r h_s (\g(s+r))
    &= \phi(\G_r) \phi'(\G_r) \g'(s+r)/\G_r.    \end{align*}
If $0 <r' < r <R$, 
     $$\abs{ \G_{r}- \G_{r'} } = \abs{(\g(s+r) - \g(s+r'))/( \G_{r}+\G_{r'} )} \leq c \abs{r-r'} /\sqrt{r},$$
     since $\G_r \geq \sqrt{4r/5}$ (see \eqref{ineq_2}).
     Now we choose furthermore $0< R_0 \leq R$ such that for all $s$, the convex hull of $\{\G_{r} ; r\leq R_0 \}$ is contained in $ \varphi_s(\nabla)$. Thus for every $r, r' \leq R_0$, $t\in [0,1]$, the segment $[t\G_{r}, t\G_{r'}]$ is in $\varphi_s (\nabla)$. Hence 
     \begin{align*}
      & \abs{\partial_r h_s (\g(s+r)) -  \partial_r h_s (\g(s+r'))} \\
      \leq & \abs{\phi(\G_r) \phi'(\G_r) \g'(s+r)/\G_r - \phi(\G_{r'}) \phi'(\G_{r'}) \g'(s+r')/\G_{r'}} \\
      \leq &C \left( \abs{\phi(\G_r)/\G_r - \phi(\G_{r'})/\G_{r'}} + \abs{\phi'(\G_r) - \phi'(\G_{r'})} + \abs{\g'(s+r) - \g'(s+r')}\right) \\
      \leq & C_3 \left[ \abs{\G_r- \G_{r'}}  (\abs{\G_r} + \abs{\G_r - \G_{r'}})^{2\b-1}+ \abs{r -r'}^{\b} \right] \\
      \leq & C_4 \left[\frac{\abs{r-r'}}{\sqrt r} (\sqrt r + \frac{\abs{r-r'}}{\sqrt r} )^{2\b-1} + \abs{r -r'}^{\b}\right] \\
     \leq  & C_4 \left[\frac{\abs{r-r'}}{\sqrt r} (2 \sqrt r)^{2\b-1} + \abs{r -r'}^{\b}\right]
      \leq  C_2 \abs{r-r'}^\b,
      \end{align*}
where all constants do not depend on $s$. We also used the fact that $\abs{r-r'} \leq \abs r$, and $r^{\b-1} \leq \abs{r-r'}^{\b-1}$ since $1/2 < \b <1$.
  
  The case $\b = 1$ is similar.
\end{proof}

\subsection{The driving function of the initial bit of the curve.}\label{sec_driving_init}

In this subsection we study the driving function of $\eta$ in a neighborhood of $0$. By comparing to an affine line (Corollary~\ref{cor_Re}, Lemma~\ref{lem_driver_comparison}), we deduce that $W_t$ is bounded above by constant times $\Re \eta(t)$ that is again comparable to $\Im(\eta (t)) \sqrt t^{2\b} \approx t^{\b +1/2}$ (Lemma~\ref{lem_vertical_curve}).

  \begin{lemma}[{\cite[Sec.~4.1]{KNK2004exact}}]\label{lem_straight}
     Let $0 \leq \t \leq \pi /4 $. There exists $k = k(\t) \leq  (16/\sqrt{3} \pi)\t$ such that the straight line $\eta = \{r e^{i (\pi/2 - \t)}, r \geq 0\}$ has the Loewner driving function $t\mapsto  k(\t) \sqrt{t}$,  and the capacity parametrized line $(\eta(t))_{t \geq 0}$ satisfies
     $$\eta(t) = B(k)\sqrt{t},$$
     where $\abs{B(k)}\geq 2$ and $\abs{B(k)} \to 2$ as $\t \to 0$. 
     \end{lemma}
     
     \begin{proof}
     From the explicit computations in \cite{KNK2004exact}, we have that the Loewner curve $\eta$ generated by $t\to k\sqrt{t}$ is the ray with argument $\pi/2 - \t(k)$, where 
     $$\theta (k) = \frac{\pi}{2} \frac{k}{\sqrt{k^2+16}}.$$
     The capacity parametrization of $\eta$ is also explicit:
     $$\eta(t) = B(k) \sqrt{t},$$
     where 
     \begin{align*}
     B(k) &= 2 \left(\frac{\sqrt{k^2+16} +k}{\sqrt{k^2+16} -k}\right)^{\dfrac{k}{2\sqrt{k^2+16}}} \exp(i(\pi/2-\t(k))) \\ 
     &= 2  \left(\frac{\pi/2 +\t(k)}{\pi/2 -\t(k)}\right)^{ \t(k) /\pi} \exp\left(i(\pi/2-\t(k))\right) \\
     &=(2+O(k^2))\exp\left(i(\pi/2-\t(k))\right).
 \end{align*}
  We see that $\abs{B(k)} \geq 2$ and the claimed convergence as $k\to 0$.
 
  For every $0\leq \t \leq \pi/4$, we have
  $$k^2+16 = (\pi/2\t)^2 k^2$$
  which implies 
  $$k = 8\t/\sqrt{\pi^2 - 4\t^2} \leq (16 /\sqrt{3} \pi) \t $$
  as claimed.
  \end{proof}
  
  \begin{corollary}\label{cor_Re}
There is a universal constant $C >0$ such that for all $0 \leq \abs{x}\leq y$, the image of $x+iy$ under the mapping-out function $g$ of the segment $\eta = [0, x + iy]$ satisfies
$$ \abs{g(x+ iy)} \leq C \abs{x}. $$
\end{corollary}
\begin{proof} Without loss of generality, assume that $x \geq 0$. Let $l = \sqrt{x^2 + y^2}$, $T = cap(\eta)$, $\t = \arctan(x/y)$ and $k = k(\t)$. We know that
   $$\abs{B(k)} \sqrt{T} = \abs{x+iy} = l$$
   and therefore
   $$T = l^2 / \abs{B(k)}^2 \leq l^2/ 4. $$
   By definition of the driving function,
   $$g(x+iy) = k \sqrt{T} \leq \frac{16 \t}{\sqrt{3} \pi}  \frac{l}{2} = \frac{8}{\sqrt{3} \pi} \t l \leq \frac{8}{\sqrt{3} \pi} 2 \sin (\t) l = \frac{16}{\sqrt{3} \pi} x, $$
   where we have used $\t \leq \pi/4$.
\end{proof}

   \begin{figure}
 \centering
 \includegraphics[width=0.8\textwidth]{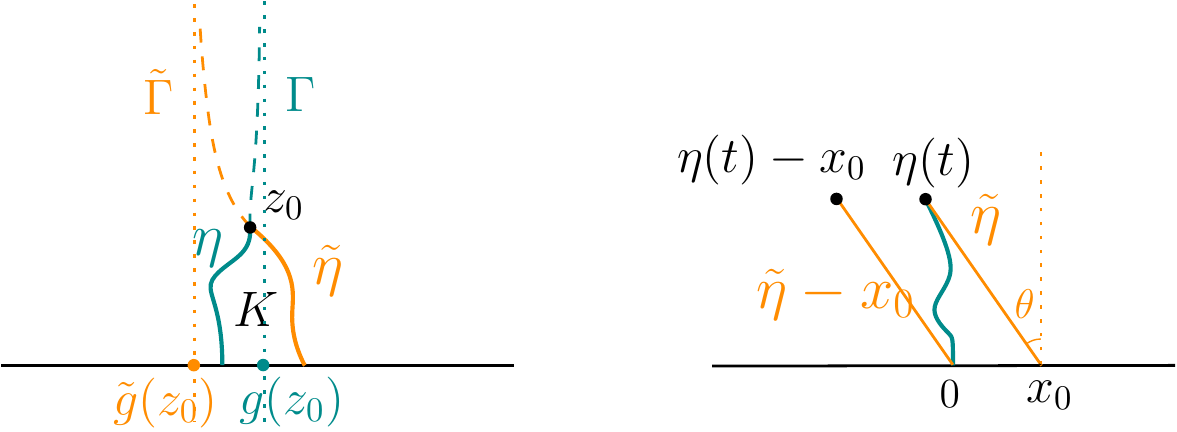}
 \caption{\label{comp} Left: The dashed line $\G$ ($\tilde \G$) is the hyperbolic geodesic between $z_0$ and $\infty$ in the domain $\m H \setminus \eta$ ($\m H \setminus \tilde \eta$) and dotted lines are their vertical asymptotes as in the proof of Lemma~\ref{lem_driver_comparison}. Right: Curves in the proof of Lemma~\ref{lem_vertical_curve}.} 
 \end{figure}

  \begin{lemma} \label{lem_driver_comparison} 
  Let $K$ be a compact $\m H$-hull whose boundary is a Jordan curve, and let $z_0 \in \partial K \cap \m H$. Denote $\eta$ (resp. $\tilde{\eta}$) the left (resp. right) boundary of $K$ connecting $\m{R}$ and $z_0$, and let $g$ and $\tilde{g}$ be their mapping-out functions. Then we have $g(z_0) \geq \tilde{g}(z_0)$.
  \end{lemma}
  \begin{proof}

 Recall that the mapping-out function $g$ of $\eta$ satisfies $g(z) = z + o(1)$. The hyperbolic geodesic $\G$ in $\m H \setminus \eta$ between $z_0$ and $\infty$ is the image of $g(z_0) + i\m{R}$ under $g^{-1}$. Hence $\G$ has the vertical asymptote   $g(z_0) + i\m{R}$. In other words, we can read off $g(z_0)$ from the geodesic.
 Let $\partial_-(\eta)$ (resp. $\partial_+(\eta)$) be the boundary of $\m H \setminus \eta$ between $z_0$ and $-\infty$ (resp. between $z_0$ and $+\infty$).
 The complement of $\G \cup \eta$ in $\m{H}$ has two connected components, $H_-(\eta)$ and $H_+(\eta)$, whose boundaries contain  $\partial_-(\eta)$ and $\partial_+(\eta)$ respectively.
 
For $z \in \m H$, let $B_z$ be a Brownian motion starting from $z$. By the conformal invariance of Brownian motion, $z \in H_-$ if and only if $B_z$ has larger probability of first hitting  $\partial_-$ than $\partial_+$. And $z \in \G$ if and only if these probabilities are equal.
It is then not hard to see that for all $z \in \tilde \G\setminus K$, we have $z\in H_-(\eta)$, where $\tilde \G$ is
the geodesic in $\m H \setminus \tilde \eta$.
In fact, the Brownian motion starting from $z$  has equal probability to hit first $\partial_-(\tilde \eta)$ or to hit $\partial_+(\tilde \eta)$. 
Besides, every sample path hitting $\partial_-(\tilde \eta)$ hits already $\partial_-(\eta)$, but not $\partial_+(\eta)$. Hence, if we stop the Brownian motion when it hits $\eta \cup \m R$, it has probability bigger than $1/2$ to hit $\partial_-(\eta)$.

By comparing asymptotes for $\G$ and $\tilde \G$, we have $\tilde{g} (z_0) \leq g(z_0)$. 
  \end{proof}

  \begin{lemma}\label{lem_cap_comparison}
     If $\g$ is an $R$-regular curve tangentially attached to $\m R_+$, then the arclength parametrization $s$ of $\g$ and the capacity parametrization $t(s)$ of $\eta = \sqrt \g$ satisfy $ s/5 \leq  t \leq s/2$, $\forall s \in [0,R]$. 
  \end{lemma}
  \begin{proof} For every $s \in [0,S]$,
     $$2t = cap(\sqrt{\g[0,s]}) \leq cap( \{z\in \m{H}, |z| \leq \sqrt s \}) = s.$$
     To see the other inequality, set
     $(X_r, Y_r) = (\Re g_r(\eta(t)), \Im g_r(\eta(t)))$ for $r \in [0,t)$. By the Loewner differential equation,
     \[\partial_r Y_r = \frac{ - 2Y_r}{(X_r-W_r)^2 + Y_r^2} \geq \frac{-2}{Y_r}.\]
     Hence 
     \[\partial_r (Y_0^2 -Y_r^2) = -2 Y_r \partial_r Y_r \leq 4\]
     so that \[4r \geq Y_0^2 - Y_r^2.\]
     We also know that $Y_{t(s)} = 0$, hence 
     $$t(s) \geq Y_0^2 /4. $$
    Since $\o(R) \leq 1/5$, we have from \eqref{ineq_1}
    $$\abs{\g(s) + s} \leq s \o(s)  \leq s/5.$$
    We conclude that
    \[Y_0 = \Im \sqrt{\g(s)} \geq \sqrt{4/5}\sqrt{s},\]
    and $ t \geq s/5$ follows.  
       \end{proof}

 \begin{lemma}\label{lem_vertical_curve}
 Using the same notation and assumption as in Lemma \ref{lem_cap_comparison},
    there exists a universal constant $c>0$, such that if $\eta$ satisfies in addition $|\pi/2 - \arg(\eta'(t))|\leq \t$ for some $0\leq \t <\pi/4$ and all $t \in [0,T]$, then the driving function $W$ is bounded by 
    \[ |W_{t(s)}| \leq c \t \sqrt{s}.
    \]
   It implies that for all $t \leq R/5$,
\begin{equation}\label{ineq_key}
\abs{\l_t} \leq c \o (5t) t^{1/2},
\end{equation}
where we recall $\o$ is the modulus of continuity  of $\g'$.
 \end{lemma}
\begin{proof}
Let $(x,y)$ denote $(\Re \eta(t), \Im \eta(t))$.
Consider the straight line segment $\tilde \eta$ that passes through $\eta(t)$ and makes an angle of $\t$ with the vertical line, as shown in Figure~\ref{comp}. 
Let  $x_0 = x + y \tan(\t)$ be the intersection of $\tilde \eta$ and $\m R$. Denote $\tilde g$ the mapping-out function of the segment $[x_0, \eta(t)]$, $g_t$  of $\eta[0,t]$ and $g$ of $[0, \eta(t) -x_0]$. 
By assumption on $\arg( \eta')$, the segment $[0,x_0]$, the curve $\eta[0,t]$ and the segment $[x_0, \eta(t)]$ form the boundary of a compact $\m{H}$-hull.
In fact, for all  $y \in (0, \Im \eta(t))$, there exists a unique point $\eta (t')$ on $\eta$  and a unique point $\tilde z$ on the segment $\tilde \eta$ with imaginary part $y$. 
It is easy to see that $\Re \eta (t') \le \Re \tilde z$.

It then follows from Corollary~\ref{cor_Re} and  Lemma~\ref{lem_driver_comparison}, $$W_t = g_t(\eta(t)) \geq \tilde g(\eta(t)) \geq g (\eta(t) - x_0) \geq C (x - x_0) = - C y \tan(\t). $$
The upper bound is similar, and we have
$$\abs{W_t} \leq C y \tan(\t) \leq  (4C / \pi )\t \sqrt{s}$$
with $C = 16/(\sqrt{3} \pi)$, where in the last inequality we have used $t= t(s)$, $y \leq \sqrt s$ and $\tan (\t) \leq 4\t /\pi$. 

In terms of $\o$, we first compute the difference between $\arg (\eta')$ and $\pi /2$:
\begin{align*}
\arg(\eta'(t)) &= \Im \log (\g'(s) / 2\sqrt{\g(s)}) = \Im \log (\g'(s)) - \Im(\log \g(s))/2\\
&= \arg(\g'(s)) - \arg(\g(s))/2.
\end{align*}
Hence from \eqref{ineq_1},
$$|\arg(\eta'(t)) - \pi/2| = |\arg(\g'(s)) -\pi - (\arg(\g(s)) - \pi)/2| \leq 2 \o(s). $$

Since  $2\o(R) \leq 2/5 < \pi/4$, we can apply the above estimate of $W$ to the interval $[0,t]$ with  $s \leq R$, $\t=2 \o(s)$, and obtain that
the driving function $\l$ of $\eta$ satisfies
\begin{equation*}
\abs{\l_t} \leq 2c  \o(s)  s^{ 1/2} \leq c' \o (5t) t^{1/2}.
\end{equation*} 
It suffices to replace $c$ by the maximum of $c$ and $c'$.
\end{proof}

\subsection{Proof of Theorem~\ref{thm_main}}\label{sec_proof_1}

Now we proceed to the proof of Theorem~\ref{thm_main}. We assume that $\g$ is a $C^{1,\b}$ curve tangentially attached to the positive real line. Without loss of generality, $\g$ is also assumed to be $R$-regular.

For $0<\b \leq 1/2$:
 We would like to compare $|\l_{t + r} - \l_t|$ to $ r^{\b+1/2}$ for every  $t \in [0,T]$ and every $r$ in a small but uniform neighborhood $[0, R_0]$ (as far as it is defined). The constant $R_0$ is chosen as in  Corollary~\ref{cor_1_1.5}.

The case $t = 0$ is already given by the inequality \eqref{ineq_key}. 
Fix $s \in (0,S]$, $t := t(s)$.   The centered mapping out function $f_s$, defined as 
  \[f_s (z) = \varphi_s^{-1} \left (\sqrt{z^2 -\g(s)} \right ), \quad f_s : \m{H} \setminus \eta[0,t] \to \m{H},\]
  maps the curve $\eta[t,T]$ to a curve $\tilde{\eta}$ whose driving function is $\tilde \l_r = \l_{t+r} - \l_t$, see Figure~\ref{figure1}.
    Since $f_s (z) = \sqrt{h_s(z^2)}$, by Corollary~\ref{cor_1_1.5}, $\tilde \g = \tilde{\eta}^2$, reparametrized by arclength, is a $C^{1,\b}$ curve: thus for $ r \leq R_0$,
   \begin{align*}
   \abs{\tilde \gamma'(r) + 1 } &\leq C_2 r^\b,  & \text{if }  0< \b < 1/2; \\
   \abs{\tilde \gamma'(r) + 1 } &\leq C_2 r^\b \log (1/r),   & \text{if } \b =1/2. 
   \end{align*}
   Here $R_0$ and $C_2$ depend on $\b, M,S, \norm{\g}_{1,\b}$, but are uniform in $s \in [0,S]$. By taking a perhaps smaller $R_0$, such that the modulus of continuity of $\tilde \g'$ at $R_0$ is less than $1/5$,
   inequality \eqref{ineq_key} in Lemma~\ref{lem_vertical_curve} applies again to  $\tilde \l$. For $r \leq R_0/5$,
 \begin{align*}
  \abs{\l_{t+r} - \l_t} &\leq c C_2 (5r)^{\b} r^{1/2} := C r^{\b +1/2} \hspace{1.5cm} &\text{if } 0 < \b < 1/2;\\
  \abs{\l_{t+r} - \l_t} &\leq c C_2 (5r)^{1/2} \log(1/5r) r^{1/2}\leq C r \log(1/r)  &\text{if } \b = 1/2,
  \end{align*}
  where $C$ depends only on the global parameters of $\g$ and on $\norm{\g}_{1,\b}$.

For $\b > 1/2$:
Since we expect that the curve has $C^1$ driving function, it is natural to compute directly the derivative of $\l$. Actually it is a multiple of $L_s$ (defined in Proposition~\ref{prop_varphi}) which equals to the second derivative at $0$ of the uniformizing map $\mu_s$ (Corollary~\ref{cor_Ls}). 
A similar result has been observed in \cite[Lem.~6.1]{lindtran2014regularity} in a more general setting, with higher order of derivatives of $\l$. 
Here we reproduce a simple proof for the first derivative for the readers' convenience. We first prove a lemma, to see how the driving function changes under a conformal transformation. The proof is standard, the same computation appears also in the study of the conformal restriction property \cite[Sec.~5]{lawler2003conformal}.

Let $\nu$ be a conformal map on a neighborhood $D$ of $0$, such that $\nu(0) = 0$, $\nu(D\cap \m H) \subset \m H$ and $\nu(D\cap \m R) \subset \m R$. 
Let $\eta$ be a curve in $\m H$ driven by $W$ such that $\eta$ is contained in $D$. Define $\tilde \eta (t) := \nu(\eta (t))$.
Let $g_t$ and $\tilde g_t$ denote the mapping-out function of $\eta[0,t]$ and $\tilde \eta[0,t]$ respectively, and $\phi_t = \tilde g_t \circ \nu \circ g_t^{-1}$ denote the conformal map that factorizes the diagram (Figure~\ref{figure2}).
Note that $\phi_0 = \nu$, 
and define $\tilde W_t = \phi_t (W_t)$.

\begin{lemma}\label{lem_mobius}
Assume that $\abs{W_t / t}$ is bounded. Then  we have
$$ \abs{\tilde W_t  - \nu'(0) W_t + 3 \nu''(0) t} / t \xrightarrow[]{t \to 0} 0.$$
\end{lemma}

 \begin{figure}
 \centering
 \includegraphics[width=0.9\textwidth]{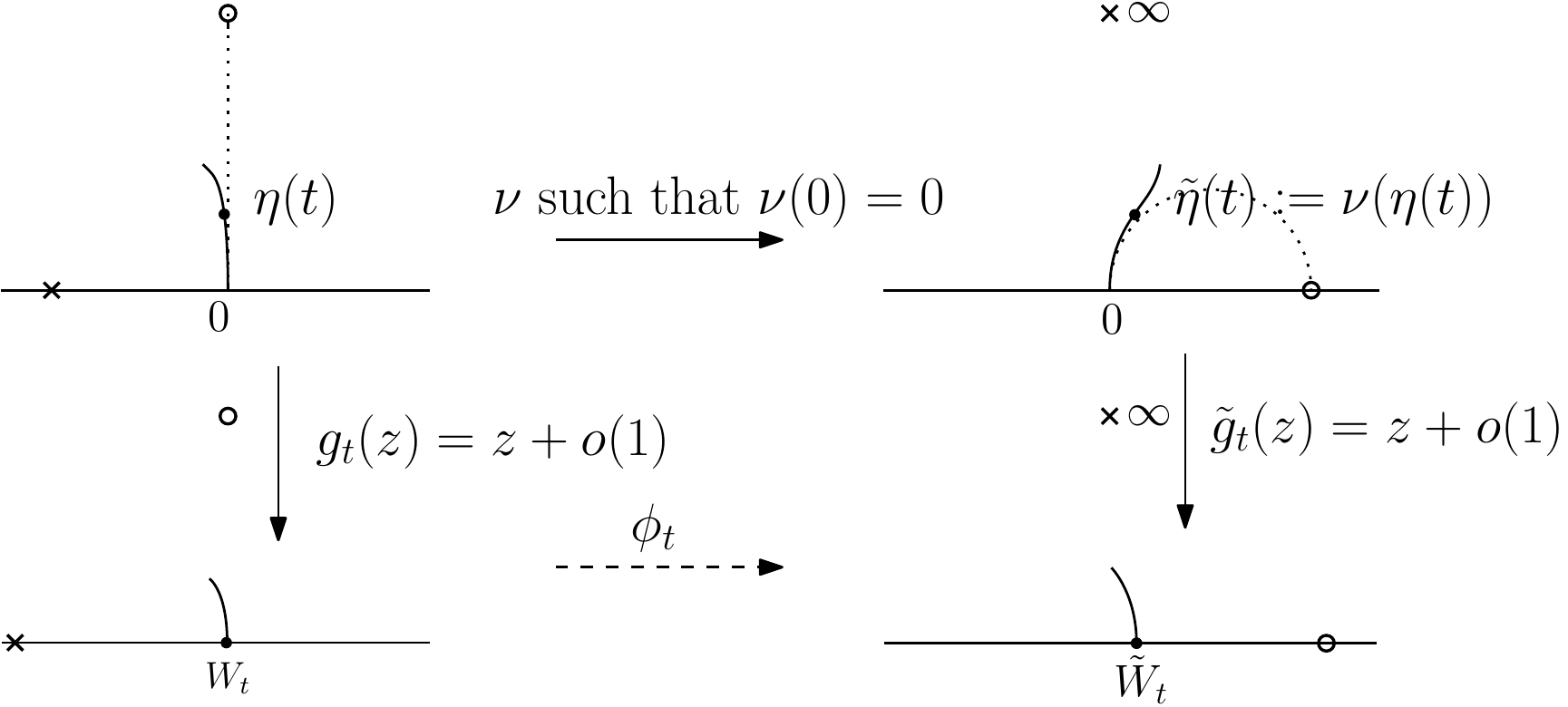}
 \caption{\label{figure2} The conformal map $\phi_t$ factorizes the diagram.} 
 \end{figure}

\begin{proof}  
Notice that $\tilde \eta (t)$ is not capacity-parametrized. 
Let $2 a(t)$ denote the capacity of $\tilde \eta [0,t]$. We have then $a'(t) = [\phi_t'(W_t)]^2$.

It is not hard to see that for any continuous driving function $W$, the map $t \mapsto \phi_t^{(n)}(z)$ is at least $C^1$ for all  $n \geq 0$ and all $z \in \overline{\m H}$ for which  $\phi_t(z)$ is well-defined
(when $z \in \m R$, this follows from the Schwarz reflection principle). 
We deduce that $r \mapsto \phi'_r(W_r)$ and $r \mapsto \phi''_r(W_r)$ are both continuous as well as any higher order derivatives of $\phi_r$ evaluated at $W_r$ (and differentiable if $W$ is so).

From that, it is not hard to see that there exists $t_0, \d>0$, and $C>0$, such that for all $t \leq t_0$ and $\abs{z} \leq \d$, we have $\abs{R(z)} \leq C\abs{z}^3$ and $\abs{R'(z)} \leq C\abs{z}^2$, where $R$ is defined as 
$$R(z) = \phi_t(W_t + z) - \tilde W_t - z \phi_t'(W_t)  - z^2 \phi_t''(W_t) / 2,$$
and
$$R'(z) = \phi'_t(W_t + z) - \phi'_t(W_t) - z \phi''_t(W_t).$$
For $z \in \m H$,
\begin{align*}
    \partial_r \phi_r(z) &= \partial_r \tilde g_r \circ \nu \circ g_r^{-1} (z) \\
    &= a'(r) \partial_a \tilde g_r (\nu \circ g_r^{-1} (z)) + \tilde g_r'(\nu \circ g_r^{-1} (z))  \nu'(g_r^{-1} (z)) \partial_r g_r^{-1} (z)  \\
    & =  \frac{2a'(r)}{\phi_r(z) - \tilde W_r} - \frac{2\phi_r'(z)}{z- W_r},
\end{align*}
where we have used 
$$\partial_r g_r^{-1} (z) = \frac{-2 (g_r^{-1})' (z)}{z - W_t}.$$
For simplicity of notation, we will omit the argument $W_t$ in the following computation. 
\begin{align*}
\partial_r \phi_r(z + W_r) &= \frac{2 (\phi_r')^2}{z\phi_r' +z^2 \phi''_r/2+R(z) } - \frac{2(\phi'_r + z \phi''_r +R'(z))}{z} \\
& = \frac{2\phi_r'}{z} \cdot \frac{1 - (1+z \phi''_r/2\phi_r'+R(z)/z\phi_r') (1 +z \phi_r''/\phi_r' + R'(z) /\phi_r')}{1+z \phi''_r/2\phi_r'+R(z)/z\phi_r'} \\
& = -3 \phi''_r (W_r) + T_r(z),
\end{align*}
with $T_r(z)/z$ bounded on $(z,r) \in \mc O \times [0,t_0]$, where $\mc O$ is a small neighborhood of $0$. Thus $T_r (z) \to 0 $ as $z \to 0$ uniformly in $r \in [0,t_0]$. 
\begin{align*} 
     &\tilde W_t  - \nu'(0) W_t + 3\nu''(0) t   \\
   =  &\lim_{z\to W_t} \phi_t (z)  - \nu'(0) W_t +3\nu''(0) t \\
   = & - \nu'(0) W_t+ \nu(W_t)  + \lim_{z\to W_t} \int_0^t \partial_r \phi_r(z) dr + 3\nu''(0) t  \\
   =&  \int_0^{W_t} (\nu'(s)- \nu'(0))ds + \int_0^t 3(\nu''(0)-\phi''_r(W_r)) + T_r(W_t -W_r) dr. 
\end{align*}
Since $W_t/t$ is bounded, the first integral divided by $t$ converges to $0$ as $t \to 0$. The second integral divided by $t$ converges to $0$ since the integrand converges uniformly to $0$ as $t \to 0$, which concludes the proof.
\end{proof}

In particular, if $W$ is differentiable at $0$, then the derivative with respect to the capacity of $\tilde {\eta}$ also exists at $0$, and 
\begin{equation}\label{eq_derivative}
\partial_a {\tilde W}|_{a=0} = \lim_{t \to 0} a'(0)^{-1} \partial_t \tilde W|_{t=0} = \dot W_0 /\nu'(0) - 3 \nu''(0)/\nu'(0)^2,
\end{equation}
as $a'(0) = \nu'(0)^2$.

\begin{corollary} \label{cor_Ls}
If $\b >1/2$, the driving function $W$ is right-differentiable. Moreover $\partial_{t+} W_t = 3L_s$, where $t(s) = t$ and $L_s$ is defined in Proposition~\ref{prop_varphi}.
\end{corollary}

\begin{proof}  (See Figure~\ref{figure1})  We use the notation as in Corollary~\ref{cor_1.5_2} and let $\nu = \mu_s^{-1}$. From Corollary~\ref{cor_1.5_2}, $\nu$ maps a Loewner chain driven by a certain function $V$ to $\tilde{\eta}$. This Loewner chain is the square root of a $C^{1,\b}$ curve.
By inequality \eqref{ineq_key} and the same proof as for the case $\b \leq 1/2$, we have 
$$\abs{V_t} \leq C t^{\b+1/2}$$
for small $t$,
in particular $\dot V(0) = 0$ as $\b >1/2$. Recall that 
the driving function of $\tilde {\eta}$ is $\tilde W_h = W_{t+h} - W_t$.
By Lemma~\ref{lem_mobius} and equation \eqref{eq_derivative},
  we have
\[ \partial_{t+} W_t = \dot V (0)  -3 \nu''(0) = 3 \mu_s''(0) = 3L_s,\]
where we have used $\nu'(0) = 1$.
\end{proof}

In particular $\dot{W}_0 = 0$. Notice that the above corollary only deals with the right derivatives of $W$. In the following lemma, we will see that $L$ is continuous. By elementary analysis, continuous right-derivative implies that $W$ is $C^1$, with the actual derivative $3L$.  See for example \cite[ Lem.~4.2]{lawler2008conformally} for a proof. 
Notice also that $3L_s$ depends only on $\g[0,s]$, it is then not surprising that it also gives the left derivative of $W$.

\begin{lemma} \label{lem_Ls_bound}
  There exists $C'$ and $C''$  such that for all $s \in [0,R]$, 
     $$ \abs{L_s} \leq C' \left( \frac{\o(s)}{\sqrt{s}} +  \int_0^{\sqrt{s}} \frac{\o(r^2)}{r^{2}}  dr \right) \leq C''
     \begin{cases}
     s^{\b - 1/2}, &\text{ if $\g$ is $C^{1,\b}$ }, \\
     s^{1/2} \log(1/s)  &\text{ if $\g$ is weakly $C^{1,1}$}.
     \end{cases}$$
\end{lemma}
\begin{proof}
    We use the explicit expression for $L_s$. From equation \eqref{eq_Ls} in Proposition~\ref{prop_varphi}, 
    \[L_s = \frac{1}{\pi} \int_{-\infty}^{\infty} \frac{w_s(r) }{r^2} dr, \]
    where $w_s(r) = \Im \log (\varphi_s'(r)) -\Im \log (\varphi_s'(0))$. 
Since $s \leq R \leq 1/2$, from Lemma~\ref{lem_u} and a similar proof of Lemma~\ref{lem_boundary_regularity}, we easily deduce that
$$\abs{w_s(r)} \leq C ( \o(r^2) \wedge \o(s)).$$

This yields
\begin{align*} 
  \abs{L_s} &\leq \frac{2C}{\pi}  \left( \o(s) \int_{ \sqrt{s}}^{\infty} \frac{1}{r^2} dr +  \int_0^{\sqrt{s}} \frac{\o(r^2)}{r^{2}} dr \right) \\
  & =  C'  \left( \frac{\o(s)}{ \sqrt{s}} + \int_0^{\sqrt{s}} \frac{\o(r^2)}{r^{2}}  dr \right).
  \end{align*}

In particular, when $\o(\d) = \norm{\g}_{1,\b} \d^\b$, 
\[\abs{L_s} \leq C' \norm{\g}_{1,\b}\left(s^{\b -1/2}+ \frac{s^{\b - 1/2}}{2\b -1}\right) =  C'' s^{\b - 1/2}. \]
When $\o(\d) =  \d \log(1/\d),$
\begin{align*} 
  \abs{L_s} &\leq C' \left(s^{1/2} \log(1/s) - 2 \int_0^{\sqrt s} 
  \log(r) dr \right) \\
  &=C' \left(s^{1/2} \log(1/s) - 2 \left[x\log(x) - x\right]_0^{\sqrt s} \right)\\
  & \leq C'' s^{1/2} \log(1/s),
  \end{align*}
where $C''$ does not depend on $s$ but only on $\b, R, S$ and $\norm{\g}_{1,\b}$.
\end{proof}

Now Theorem~\ref{thm_main} for $ 1/2 < \b \leq 1$ follows directly from Corollary~\ref{cor_1.5_2}, Corollary~\ref{cor_Ls} and Lemma~\ref{lem_Ls_bound}.

\section{Comments}\label{sec_comments}
\subsection{The sharpness of Theorem~\ref{thm_main}}\label{subsec_sharpness}
 As we already argued in the introduction, as the converse of Theorem~\ref{thm_wong}, 
Theorem~\ref{thm_main} is sharp in the range $\b\in (0,1/2)\cup(1/2,1)$.
In fact, for those values of $\b$, the regularity of the driving function implies  (Theorem~\ref{thm_wong}) capacity regularity of the generated curve which implies arclength regularity of the curve. Then by Theorem~\ref{thm_main}, it implies again the regularity of the driving function, where the regularities are taken accordingly with a shift of $1/2$ as in both theorems.

The example in \cite[Sec.~7.2]{lindtran2014regularity} shows that the
driving function of a $C^{1,1/2}$- curve need not be in $C^1$ but 
may only be in $C^{0,1}.$ Thus in the case
$\beta=1/2$, our theorem is sharp up to the logarithmic term. 
Similarly, \cite[Sec.~7.1]{lindtran2014regularity} provides an example of a $C^{1,1}-$
curve whose driving function is $C^{1,1/2}.$ We do not know if our result
can be improved by removing the term ``weakly'' in the cases $\beta=1/2$
and $\beta=1.$

The case of higher regularity requires the consideration of higher angular derivatives of the uniformizing map $\varphi_s$ at $0$. Nevertheless, we believe that the proof of the natural generalization of Theorem~\ref{thm_main} should be in the same spirit. 
Since the focus of this paper is on the Loewner energy, we refrain from discussing the converse of Theorem~\ref{thm_lind_tran} in full generality.

\subsection{Finite energy and slow spirals} Finite energy curves are rectifiable 
and therefore have tangents on a set of full length and full harmonic measure. However, 
we sketch an example showing that finite energy loops need not have tangents everywhere: Pick a sequence $\eps_k$ such that $\sum_k \eps_k$ diverges but $\sum_k \eps_k^2$ converges, and consider a sequence $r_k\to0$ of scales. By \cite{wang2016}, the chordal energy minimizing curve $\gamma_k$ from $0$ to $z_k=r_k e^{i (\pi/2 +\eps_k)}$ in $\HH$ has energy $I_k = -8 \ln \sin(\pi/2+ \eps_k) \sim 4\eps_k^2$ so that the conformal concatenation $\G_k$ (whose mapping-out function is $G_k = g_k\circ g_{k-1}\circ...\circ g_1$ and $g_i$ is the mapping-out function of $\g_i$)
has uniformly bounded energy. Denote $\alpha_k$ the tangent angle of the tip of $\G_k.$
Since $G_k$ behaves like the square-root map near the tip of $\G_k$, given $r_1,r_2,...,r_k$ we have
$\alpha_{k+1}=\alpha_k + 2\eps_k + o(1)$ as $r_{k+1}\to0.$ Thus the sequence $r_n$ can be chosen inductively in such a way that $\alpha_n \geq \alpha_1+\sum_1^{n-1} \eps_k$ for all $n.$ Consequently, the limiting curve $\Gamma=\cup_k\Gamma_k$ has an infinite spiral at its tip and does not possess a tangent there.

\subsection{Consequences of Theorem~\ref{thm_main_2}}
\label{isotopy}
      
   \begin{figure}[ht]
 \centering
 \includegraphics[width=0.7\textwidth]{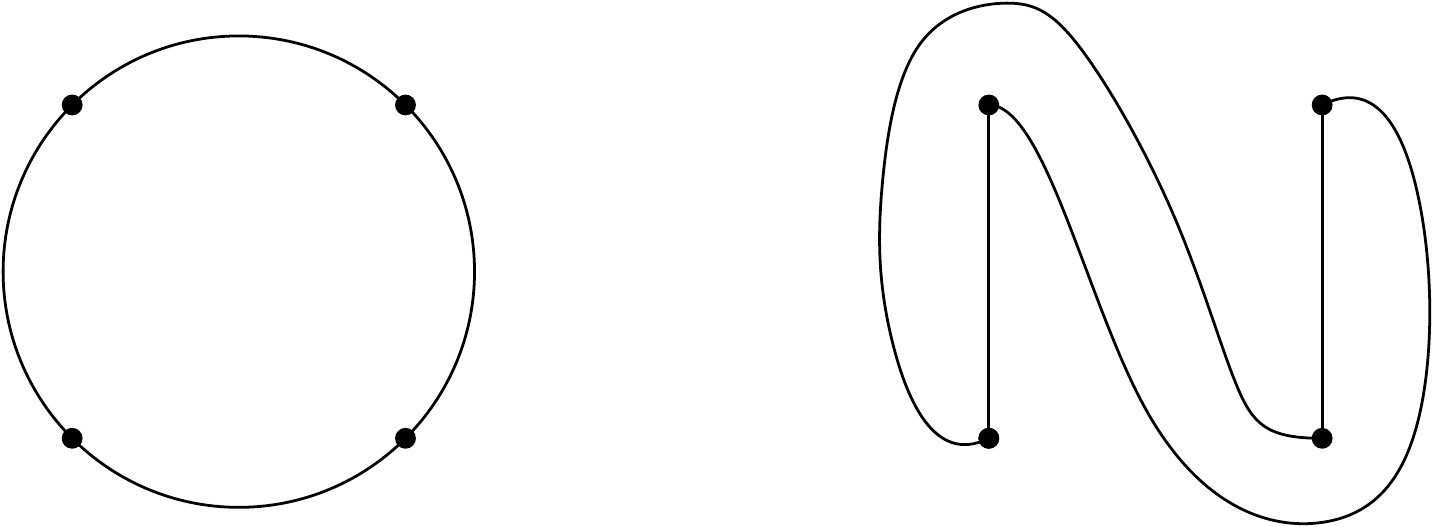}
 \caption{\label{fig_4points} Two non-isotopic loops passing through four points in the same order.} 
 \end{figure}
 
Proposition \ref{prop_optimal_loop} and Corollary \ref{cor_minimal_energy} can be generalized as follows:
As before, fix a collection of distinct points $\ad{z} = (z_0, z_1, z_2,\cdots, z_n)$ 
and consider curves $\gamma$ visiting these points in order. 
Figure \ref{fig_4points} shows two such curves, visiting
the same points in the same order, that cannot be continuously deformed into each other while fixing the 
points and keeping the curves simple. 
For three distinct points (the case $n=2$) there is only one isotopy class, and the minimal energy is 0.
For four or more points, there are always countably infinite many classes.
The proof of Proposition \ref{prop_optimal_loop} can easily be modified to show that
each of these isotopy classes of curves contain at least one loop energy minimizer. More precisely,
fix a Jordan curve $\gamma_0$ compatible with $\ad{z}$, 
denote $\mc{L}(\ad{z},\gamma_0)$ the set of all Jordan 
curves $\gamma_1$ for which there is a homotopy $\gamma_t$ relative $\ad{z}$ through homeomorphisms (that is, in addition to the joint continuity of $\gamma_t(s),$ we require that each $\gamma_t$ is a Jordan curve, and that $\gamma_t(\gamma_0^{-1}(z_j))=z_j$ for all $j=0,1,...,n$ and all $0\leq t\leq 1$)
and set
$$I^L(\{\ad{z}, \gamma_0\}) := \inf_{\g \in \mc{L}(\ad{z},\gamma_0)} I^L(\g),$$
where we have dropped the root in the above expression since the loop energy is root-invariant.

Then we have:
\begin{proposition}\label{prop_optimal_loop_isotopy} There exists $\g \in \mc{L}(\ad{z},\gamma_0)$ such that 
$I^L(\g) =  I^L(\{\ad{z}, \gamma_0\})$, and every such $\g$ is at least weakly $C^{1,1}$. 
\end{proposition}

It seems reasonable to believe that the minimizer in each class is unique.
In any case, every minimizer has the property that the arc between consecutive points is a hyperbolic geodesic in the 
complement of the rest of the loop as in the proof of Proposition \ref{prop_optimal_loop}.

\medskip

{\bf Acknowledgements}    We would like to thank Wendelin~Werner for discussions on the loop Loewner energy, Fredrik Viklund for his very useful comments on the first draft, Huy Tran for discussions on the quasiconformality of finite energy loops, Don Marshall for his contribution to the study of geodesic pairs, and Brent Werness for his permission to include his simulation of energy minimizing curves. We also thank the referee for very helpful comments. This work was supported by the National Science Foundation [DMS1362169, DMS1700069 to S.R.]; and Swiss National Science Foundation [SNF155922 and its mobility grant to Y.W.].


\begin{thebibliography}{0}




\bibitem{astala2008elliptic}
Astala, K., Iwaniec, T., Martin, G. (2008):
 \textit{ Elliptic Partial Differential Equations and Quasiconformal
  Mappings in the Plane.}
  Princeton University Press. 

\bibitem{DeBranges1985}
De~Branges, L. (1985):
 \textit{ A proof of the {Bieberbach} conjecture.}
  Acta Math., \textbf{154}, no.~1-2, 137--152. 


\bibitem{Dubedat2007commutation}
Dub{\'e}dat, J. (2007):
 \textit{ Commutation relations for {Schramm-Loewner} evolutions.}
  Commun. Pure. Appl. Math.,
  \textbf{60}, no.~12, 1792--1847. 

\bibitem{Dubedat2009partition}
Dub{\'e}dat, J.  (2009):
\textit{SLE and the free field: Partition functions and couplings.}
J. Amer. Math. Soc., \textbf{2}, no.~4, 995--1054.

\bibitem{EE2001}
Earle, C. J., Epstein, A. L. (2001):
\textit{Quasiconformal variation of slit domains.}
Proc. Amer. Math. Soc. \textbf{129}, 3363-3372. 


\bibitem{friz2015}
Friz, P., Shekhar, A. (2017):
  \textit{On the existence of SLE trace: finite energy drivers and
  non-constant $\kappa$.}
Probab. Theory Relat. Fields, \textbf{169}, 1-2.
 
 \bibitem{GM}
Garnett, J., Marshall, D. (2005):
\textit{Harmonic measure.} Cambridge Univ. Press, Cambridge.

\bibitem{KNK2004exact}
Kager, W., Nienhuis, B., Kadanoff, L.P.  (2004):
\textit{Exact solutions for Loewner evolutions.}
J. Stat. Phys \textbf{115}, 805--822. 

\bibitem{lawler2008conformally}
Lawler, G. (2008):
 \textit{ Conformally invariant processes in the plane.}
   Amer. Math. Soc. 
   
   
\bibitem{lawler2009part}
Lawler, G. (2009):
\textit{Partition Functions, Loop Measure, and Versions of SLE. }
J. Stat. Phys., \textbf{134}, 813--837. 

\bibitem{lawler2001values}
Lawler, G., Schramm, O., Werner, W. (2001):
 \textit{ Values of {Brownian} intersection exponents, {I}: {Half-plane}
  exponents.}
  Acta Math., \textbf{187}, no.~2, 237--273.

\bibitem{lawler2003conformal}
Lawler, G., Schramm, O., Werner, W. (2003):
 \textit{ Conformal restriction: the chordal case.}
  J. Amer. Math. Soc., \textbf{16}, no.~4, 917--955.
  

\bibitem{lehto2012univalent}
Lehto, O. (2012):
 \textit{ Univalent functions and Teichm{\"u}ller spaces.}
  Springer. 

\bibitem{lehto1973quasiconformal}
Lehto, O., Virtanen, V. (1973):
 \textit{ Quasiconformal mappings in the plane.}
  Springer.

\bibitem{lind2005sharp}
Lind, J. (2005):
 \textit{ A sharp condition for the {Loewner} equation to generate slits.}
  Ann. Acad. Sci. Fenn.  Math., \textbf{30}, 143--158.

\bibitem{lind2010collisions}
Lind, J., Marshall, D., Rohde, S. (2010):
 \textit{ Collisions and spirals of {Loewner} traces.}
  Duke Math. J., \textbf{154}, no.~3:527--573.
  
\bibitem{lindtran2014regularity}
Lind, J., Tran, H. (2016):
\textit{Regularity of Loewner curves.}
 Indiana Univ. Math. J., \textbf{65}, 1675--1712.  

\bibitem{Loewner1923}
Loewner, K. (1923):
 \textit{ Untersuchungen \"uber schlichte konforme Abbildungen des
  Einheitskreises I.}
  Math. Ann., \textbf{89}, no.~1-2, 103--121.

\bibitem{marshall2005loewner}
Marshall, D., Rohde, S. (2005):
 \textit{ The {Loewner} differential equation and slit mappings.}
  J. Amer. Math. Soc., \textbf{18}, no.~4, 763--778.

\bibitem{MRW2017geodesic}
Marshall, D., Rohde, S., Wang, Y.:
\textit{Hyperbolic geodesic graphs.}
In preparation.

\bibitem{pomm1992boundary}
Pommerenke, C. (1992):
\textit{Boundary Behaviour of Conformal Maps.}
Springer, Grundlehren Math. Wiss.

\bibitem{Schramm2000}
Schramm, O. (2000):
\textit{Scaling limits of loop-erased random walks and
uniform spanning trees. } Israel J. Math.
\textbf{118}, 221--288.
  

\bibitem{wang2016}
Wang, Y. (2016):
\textit{The energy of a deterministic Loewner chain: Reversibility and interpretation via SLE$_{0+}$.}
To appear in J. Eur. Math. Soc.


\bibitem{W2}
Wang, Y. (2018):
\textit{Equivalent descriptions of the Loewner energy.}
Preprint. 

\bibitem{war1932}

Warschawski, S. (1932): 
\textit{\"Uber das Randverhalten der Ableitung der Abbildungsfunktion bei konformer Abbildung. } (German) Math. Z. \textbf{35} no.~1, 321--456. 

\bibitem{werner2004st_flour}
Werner, W. (2004):
 \textit{ Random planar curves and {Schramm-Loewner} evolutions.}
  Ecole d'Et\'e de Probabilit\'es de Saint-Flour XXXII, Lectures Notes in Math. Springer, 1840,
  107--195.



\bibitem{wong2014}
Wong, C. (2014):
\textit{Smoothness of Loewner slits.}
Trans. Amer. Math. Soc., \textbf{366}, 1475-1496.

\bibitem{zhan2008}
Zhan, D.  (2008):
\textit{Reversibility of chordal {SLE}.}
Ann. Probab., \textbf{36}, 1472--1494.




\end{thebibliography}
\end{document}